\documentclass[12pt,reqno]{amsart}
\usepackage{latexsym}
\usepackage{amssymb}
\usepackage{mathrsfs}
\usepackage{amsmath}
\usepackage{comment}
\usepackage{fancybox,color}
\usepackage{enumerate}
\usepackage[latin1]{inputenc}
\usepackage{soul} 
\usepackage[colorlinks=true, linkcolor=blue, citecolor=blue]{hyperref}

\newcommand{\ch}[1]{{\mbox{\raise 1pt\hbox{\large$\chi$}}}_{\lower 1pt\hbox{$\scriptstyle #1$}}}

\usepackage[left=2.1cm,top=2.3cm,right=2.1cm]{geometry}

\def\1{\raisebox{2pt}{\rm{$\chi$}}}

\newtheorem{theorem}{Theorem}[section]
\newtheorem{corollary}[theorem]{Corollary}
\newtheorem{lemma}[theorem]{Lemma}

\theoremstyle{definition}
\newtheorem{definition}[theorem]{Definition}
\newtheorem{remark}[theorem]{Remark}
\newtheorem{example}[theorem]{Example}

\newcommand{\R}{{\mathbb R}}

\newcommand{\N}{{\mathbb N}}

\newcommand{\Ha}{{\mathcal H}}

\newcommand\cp{\operatorname{cap}}
\newcommand\diam{\operatorname{diam}}

\def\1{\raisebox{2pt}{\rm{$\chi$}}}

\newcommand{\Lip}{\operatorname{Lip}}

\def\vint_#1{\mathchoice%
        {\mathop{\kern 0.2em\vrule width 0.6em height 0.69678ex depth -0.58065ex
                \kern -0.8em \intop}\nolimits_{\kern -0.4em#1}}%
        {\mathop{\kern 0.1em\vrule width 0.5em height 0.69678ex depth -0.60387ex
                \kern -0.6em \intop}\nolimits_{#1}}%
        {\mathop{\kern 0.1em\vrule width 0.5em height 0.69678ex
            depth -0.60387ex
                \kern -0.6em \intop}\nolimits_{#1}}%
        {\mathop{\kern 0.1em\vrule width 0.5em height 0.69678ex depth -0.60387ex
                \kern -0.6em \intop}\nolimits_{#1}}}
\def\vintslides_#1{\mathchoice%
        {\mathop{\kern 0.1em\vrule width 0.5em height 0.697ex depth -0.581ex
                \kern -0.6em \intop}\nolimits_{\kern -0.4em#1}}%
        {\mathop{\kern 0.1em\vrule width 0.3em height 0.697ex depth -0.604ex
                \kern -0.4em \intop}\nolimits_{#1}}%
        {\mathop{\kern 0.1em\vrule width 0.3em height 0.697ex depth -0.604ex
                \kern -0.4em \intop}\nolimits_{#1}}%
        {\mathop{\kern 0.1em\vrule width 0.3em height 0.697ex depth -0.604ex
                \kern -0.4em \intop}\nolimits_{#1}}}

\newcommand{\dist}{\operatorname{dist}}

\title[Capacities and density conditions in metric spaces]{Capacities and density conditions in metric spaces}

\author[J.\! Canto]{Javier Canto}
\address[J.C.]{Facultad de Ciencia y Tecnolog\'ia, Universidad del Pa\'is Vasco / Euskal 
Herriko Unibertsitatea (UPV/EHU), Departamento de Matem\'aticas, Apartado 644, 48080 Bilbao, Spain}
\email{javier.canto@ehu.eus}

\author[L. Ihnatsyeva]{Lizaveta Ihnatsyeva}   %
\address[L.I.]{Department of Mathematics, Kansas State University, Manhattan, KS 66506, USA}
\email{ihnatsyeva@math.ksu.edu}

\author[J. Lehrb\"ack]{Juha Lehrb\"ack}   %
\address[J.L.]{Department of Mathematics and Statistics, P.O. Box 35, FI-40014 University of Jyvaskyla, Finland}
\email{juha.lehrback@jyu.fi}

\author[A. V. V\"ah\"akangas]{Antti V. V\"ah\"akangas}
\address[A.V.V.]{Department of Mathematics and Statistics, P.O. Box 35, FI-40014 University of Jyvaskyla, Finland}
 \email{antti.vahakangas@iki.fi}

\makeatletter
\@namedef{subjclassname@2020}{{\mdseries 2020} Mathematics Subject Classification}
\makeatother



\begin{document}

\begin{abstract}
We examine the relations between different capacities in the setting of a metric measure space.
First, we prove a comparability result for the Riesz $(\beta,p)$-capacity and the relative
Haj{\l}asz $(\beta,p)$-capacity, for $1<p<\infty$ and $0<\beta \le 1$,
under a suitable kernel estimate related to the Riesz potential. Then we show that
in geodesic spaces the corresponding capacity density conditions are equivalent even
without assuming the kernel estimate. In the last part of the paper, we compare
the relative Haj{\l}asz $(1,p)$-capacity to the relative variational $p$-capacity.
\end{abstract}

\maketitle

\section{Introduction}

It is well known \cite{MR0350027,MR817985} that the classical Newtonian capacity of a 
subset of the Euclidean space $\R^n$ can be characterized as the minimum of an energy functional 
in terms of the weak gradient on the Sobolev space $W^{1,2}(\R^n)$. This characterization  has led to extensions of the notion of capacity in several directions and, in particular, to the development of a theory of capacities related to Riesz and Bessel potentials; see \cite{MR1411441}. 
In many cases the
different definitions lead to comparable
capacities, see e.g.~\cite{MR1628134,MR2839008,MR1411441,MR328109,MR0276487}. 
For instance, the Riesz $(\beta,p)$-capacity $R_{\beta,p}$ %
is comparable to the classical Newtonian capacity when $\beta=1$ and $p=2$. 

In this paper, we consider various capacities 
in the setting of a metric measure space $X$,
and our main goal is to clarify the connections between these different notions. 
Analogues of Riesz potentials in metric measure spaces
have been studied for instance in \cite{MR1683160,MR1800917,MR3825765,MR2569546,MR1909289}, 
and with these it is possible to consider the corresponding Riesz capacities, see e.g.~\cite{Nuutinen2015TheRC}
and Section~\ref{s.HRcap}.
On the other hand, generalizations of the (weak) gradients to the metric setting
lead to variants of capacities 
defined in terms of Sobolev functions, such as the Newtonian capacity. 
In particular, we will use a metric space
version 
of the relative variational $p$-capacity
$\cp_p(F,\Omega)$ defined via 
$p$-weak upper gradients as in~\cite{MR1809341,MR2867756}; see Section~\ref{s.equivbeta1}.
Here $\Omega\subset X$ is a bounded open set and $F\subset\Omega$ is a closed set.
Another possibility is to use 
the so-called Haj{\l}asz gradients. These 
were first defined in~\cite{MR1401074}, and 
like Riesz potentials, they are of non-local nature,
whereas $p$-weak upper gradients are local. 
The Haj{\l}asz gradient approach to capacities in metric spaces has been used for instance  
in~\cite{MR1404091,MR1752853,MR2424909}. Following~\cite{CV2021}, we consider the relative 
Haj{\l}asz $(\beta,p)$-capacity $\cp_{\beta,p}(F,\Omega)$,
where $F$ and $\Omega$ are as above; see Section~\ref{s.HRcap}. 

After preliminary definitions and results in Sections~\ref{s.prelim} and~\ref{s.HRcap},
we begin the comparisons of capacities in Section~\ref{s.cap_are_comparable},
where we study the Riesz and Haj{\l}asz capacities of sets relative to balls.
In the special case of an Ahlfors regular metric space, the main result of that section
(Corollary~\ref{c.HCapEqvRCap}) reads as follows; see Remark~\ref{r.Qreg}.

\begin{theorem}\label{thm.RH_intro}
Let $1<p<\infty$, $0<\beta< 1$,   and $0< Q<\infty$   be such that $Q>\beta p$. 
Assume that $X$ is a complete and connected metric space
equipped with an Ahlfors $Q$-regular measure. 
Then there is a constant $C>0$ such that
\begin{equation}\label{e.comparabilityRieszHajlasz_intro}
 C^{-1} \: R_{\beta,p}(E\cap \overline{B(x,r)}) \leq \cp_{\beta,p} \big( E\cap \overline{B(x,r)},B(x,2r) \big)
 \le C\: R_{\beta,p}(E\cap \overline{B(x,r)})\,,
\end{equation}
whenever $E\subset X$ is a closed set, $x\in E$ and $0<r<(1/8)\diam(X)$. 
\end{theorem}

More generally,
the lower bound in~\eqref{e.comparabilityRieszHajlasz_intro} holds under much weaker assumptions,
while in the upper bound the Ahlfors regularity assumption can be replaced with
a reverse doubling condition and an explicit ``kernel estimate''. See 
Theorems~\ref{p.HaCap_geq_RieszCap} and~\ref{th.HCapLessRCap}, respectively,
for details.

From the point of view of analysis on general metric spaces, the assumptions for the upper
bound in~\eqref{e.comparabilityRieszHajlasz_intro} are still rather restrictive.
In Section~\ref{s.lequiv_density}, we take another approach to the comparison of capacities
under different assumptions. Namely, instead of a direct comparison between the Haj{\l}asz and Riesz
capacities, we study the equivalence of the corresponding capacity density conditions. 
For example, a closed set $E\subset X$ satisfies the Riesz $({\beta},p)$-capacity 
density condition, if there is a constant $c>0$ such that
\begin{equation*}%
R_{\beta,p}(E\cap \overline{B(x,r)})\ge c \: R_{\beta,p}(\overline{B(x,r)}) 
\end{equation*}
for all $x\in E$ and all $0<r<(1/8)\diam(E)$. The definitions 
of density conditions for other capacities are similar,
see Definitions~\ref{d.cap_density} and~\ref{d.capacity density}.

The  origins of the
(Riesz) capacity density conditions come
from  \cite{MR946438}, where the Riesz $({\beta},p)$-capacity  $R_{\beta,p}$ was used in Euclidean spaces
to define a fatness condition. %
The main result   of~\cite{MR946438} states that in $\R^n$   the Riesz capacity density condition is open ended on $\beta$ and on $p$, that is, if a set   $E\subset\R^n$ satisfies the Riesz $(\beta,p)$-capacity density condition, then $E$ also satisfies the condition   for some $q<p$ and $\alpha<\beta$. This result has found numerous applications in potential theory{ ,} in the study of Hardy inequalities and in partial differential equations, see for instance \cite{MR1302151,KLV2021,MR1386213}.

Capacity density conditions have been generalized also to metric spaces, and  
often these conditions turn out to be open-ended (self-improving) as well.
For instance, 
the $p$-capacity density condition, defined in terms of the relative variational $p$-capacity for $1<p<\infty$,
has been shown to be open-ended in~\cite{MR1869615,MR3673660} under the ``standard assumptions''
on analysis on metric spaces (see Section~\ref{s.equivbeta1}).
On the other hand, 
for the Haj{\l}asz $(\beta,p)$-capacity density condition,   the open-endedness
on both $\beta$ and $p$
was proven recently in~\cite{CV2021} in complete geodesic spaces. 

In Section~\ref{s.lequiv_density}, we show the equivalence of 
Riesz and Haj{\l}asz $(\beta,p)$-capacity density conditions
in a complete geodesic space,
see Theorem \ref{t.Riesz_and_Hajlasz}. As a tool we use another density condition,
given in terms of Hausdorff content of codimension $q$, where $0<q<\beta p$.
Under a suitable reverse doubling condition for the space $X$, such a
Hausdorff content density condition for a set $E\subset X$ implies that
$E$ satisfies the Riesz $(\beta,p)$-capacity density condition;
this follows from the results in Section~\ref{s.aux_restimate}.
On the other hand, in~\cite{CV2021} it was shown that in complete geodesic spaces
the Haj{\l}asz $(\beta,p)$-capacity density condition is equivalent to
the Hausdorff content density condition of codimension $q$, for some $0<q<\beta p$.
Taking also into account the lower bound in~\eqref{e.comparabilityRieszHajlasz_intro},
which is valid under minimal assumptions, this leads to the equivalence 
of Riesz and Haj{\l}asz $(\beta,p)$-capacity density conditions;
see Section~\ref{s.lequiv_density} for details.

Since the Haj{\l}asz capacity density conditions are known to be open-ended by~\cite{CV2021},
we obtain as a particular consequence of the above equivalence result that  
also the Riesz $(\beta,p)$-capacity density condition
is open-ended in complete geodesic spaces,
see Corollary \ref{c.riesz-improvement}.
This is an extension of the main result of~\cite{MR946438} to metric spaces.

Finally, in Section~\ref{s.equivbeta1} %
we show that the relative variational $p$-capacity $\cp_p(F,\Omega)$ and
the Haj{\l}asz $(1,p)$-capacity $\cp_{1,p}(F,\Omega)$
are comparable, for $1<p<\infty$, assuming that the metric space $X$ supports a $q$-Poincar\'e inequality
for some $1\le q < p$. Therefore we conclude that under suitable assumptions also the 
density conditions for the variational $p$-capacity, Haj{\l}asz $(1,p)$-capacity, and
Riesz $(1,p)$-capacity are equivalent; see Theorem~\ref{t.main_wide_beta=1}.

\section{Preliminaries}\label{s.prelim}

\subsection{Metric spaces}\label{s.metric}
Throughout the paper  we assume that $X=(X,d,\mu)$ is a metric measure space equipped with a metric $d$ and a 
positive complete Borel
measure $\mu$ such that $0<\mu(B)<\infty$
for all balls $B\subset X$, each of which is an open set of the form \[B=B(x,r)=\{y\in X\,:\, d(y,x)<r\}\] with $x\in X$ and $r>0$. Under these assumptions 
the space $X$ is separable,
see \cite[Proposition~1.6]{MR2867756}.
We also assume that $\# X\ge 2$ and  
that the measure $\mu$ is {\em doubling}, that is,
there is a constant $c_\mu> 1$, called
the {\em doubling constant of $\mu$}, such that
\begin{equation}\label{e.doubling}
\mu(2B) \le c_\mu\, \mu(B)
\end{equation}
for all balls $B=B(x,r)$ in $X$. 
Here we use for $0<t<\infty$ the notation $tB=B(x,tr)$.

If $X$ is  connected,  
then the  doubling measure $\mu$ satisfies  also 
the reverse doubling condition, that is,
there is a constant  $0<c_R=C(c_\mu)<1$ such that
\begin{equation}\label{e.rev_dbl_decay}
\mu(B(x,r/2))\le c_R\, \mu(B(x,r))
\end{equation}
for all $x\in X$ and all $0<r<\diam(X)/2$;
see for instance~\cite[Lemma~3.7]{MR2867756}.
Iteration of~\eqref{e.rev_dbl_decay} shows that
there exist
an exponent $\sigma>0$ and a constant $c_\sigma>0$,
both depending on $c_\mu$ only, such that the quantitative 
reverse doubling condition
\begin{equation}\label{e.reverse_doubling}
 \frac{\mu(B(x,r))}{\mu(B(x,R))} \le c_\sigma\Bigl(\frac{r}{R}\Bigr)^\sigma
\end{equation}
holds in a connected space $X$ 
for  all $x\in X$ and  all $0<r<R\le 2\diam(X)$.

The space $X$ is called Ahlfors $Q$-regular, for $Q>0$, if there is a constant $c_Q\geq 1$ such that
\begin{equation}\label{def.Q-regular}
c_Q^{-1}r^Q\le\mu(B(x,r))\le c_Qr^Q
\end{equation}
for all $x\in X$ and all $0<r<\diam (X)$.
Note that~\eqref{e.reverse_doubling}
holds in an Ahlfors $Q$-regular space with $\sigma=Q$.

By a {\em curve} we mean a nonconstant, rectifiable, continuous
mapping from a compact interval of $\R$ to $X$;  we tacitly assume
that all curves are parametrized by their arc-length.
We say that $X$ is a {\em geodesic space}, if 
every pair of points in $X$
can be joined by a curve whose length is equal to the distance between the two points.

\subsection{H\"older functions and Haj{\l}asz gradients}
Let $A\subset X$.
We say that
$u\colon A\to \R$ is  a {\em $\beta$-H\"older function,} with 
an exponent $0<\beta\le 1$ and a constant 
$0\le \kappa <\infty$, if
\[
\lvert u(x)-u(y)\rvert\le \kappa\, d(x,y)^\beta\,,\qquad \text{ for all } x,y\in A\,.
\]
The set of all $\beta$-H\"older functions $u\colon A\to\R$
is denoted by $\Lip_\beta(A)$. 
The $1$-H\"older functions are also called {\em Lipschitz functions},
and we write $\Lip(A)=\Lip_1(A)$.

The definition of the Haj{\l}asz capacities 
will be based on the
following Haj{\l}asz $\beta$-gradients (see Definition~\ref{d.varcap}). 

\begin{definition}
For each function $u\colon X\to \R$, we let 
$\mathcal{D}_H^{\beta}(u)$ 
be the (possibly empty) family of all measurable functions $g\colon X\to [0,\infty]$ such that
\begin{equation}\label{e.hajlasz}
\lvert u(x)-u(y)\rvert \le d(x,y)^\beta\big( g(x)+g(y) \big)
\end{equation}
almost everywhere, that is, there exists an exceptional set $N=N(g)\subset X$ for which $\mu(N)=0$ and
inequality \eqref{e.hajlasz} holds for every $x,y\in X\setminus N$. 
A function  $g\in\mathcal{D}^\beta_H(u)$ is called a Haj{\l}asz $\beta$-gradient of the function $u$.
\end{definition}

The following nonlocal generalization of the Leibniz rule 
is taken from \cite[Theorem 3.4]{CV2021}, see also  \cite{MR1681586}. 
The nonlocality is reflected 
by the appearance of the two global terms $\lVert \psi\lVert_\infty$ and $\kappa$ in the statement below.

\begin{lemma}\label{l.Leibniz}
 Let  $0<\beta\le 1$. 
Assume that 
 $u\colon X\to \R$ is a bounded $\beta$-H\"older function and $\psi\colon X\to \R$ is a bounded $\beta$-H\"older
function with a constant $\kappa\ge 0$.  Then $u\psi\colon X\to \R$ is a $\beta$-H\"older function  and 
\[
(g_u\lVert \psi\lVert_\infty + \kappa \lvert  u\rvert)\ch{\{\psi\not=0\}}\in\mathcal{D}_H^{\beta}(u\psi)
\]
for all $g_u\in\mathcal{D}_H^{\beta}(u)$.
Here $\{\psi\not=0\}=\{y\in X: \psi(y)\not=0\}$.
\end{lemma}

The following $(\beta,p,p)$-Poincar\'e inequality can be proved in a similar way as Theorem~5.15 in~\cite{MR1800917}.
This inequality relates the $\beta$-Haj{\l}asz gradient to the given measure. We want to emphasize that no additional assumptions on the measure are needed here.

\begin{theorem}\label{t.pp_poincare}
Let $1\le p<\infty$ and $0<\beta \le 1$,
and assume  that 
$u\in\Lip_\beta(X)$ and $g\in \mathcal{D}_H^{\beta}(u)$.
Then
\[ 
\vint_B \lvert u(x)-u_B\rvert^p\,d\mu(x)\le 2^p \diam(B)^{\beta p}\vint_B g(x)^p\,d\mu(x)\,,
\]
whenever $B\subset X$ is a ball.
\end{theorem}

Here we use the  notation 
\begin{equation}\label{e.intaverage}
u_B=\vint_{B} u(y)\,d\mu(y)=\frac{1}{\mu(B)}\int_B u(y)\,d\mu(y)
\end{equation}
for  the integral average of $u\in L^1(B)$ over a ball $B\subset X$.

\subsection{Riesz potentials}\label{s.Riesz}

Riesz potentials appear
frequently in potential analysis,
see for instance \cite{MR1411441,MR0350027}. 
 For  a nonnegative measurable function $f$  on a metric space  $X$, the 
\emph{Riesz potential of order}  $\beta>0$  in $X$  can be defined by the expression 
\[
I_\beta f(x)=\int_{X}\frac{f(y)d(x,y)^\beta}{\mu(B(x,d(x,y)))}\,d\mu(y)\,,\qquad x\in X\,;
\]
see,  for instance,   \cite{MR1683160,MR1800917,MR3825765}.

We prove  two auxiliary  lemmata  on relations between the Riesz potential $I_\beta f$ and
the  non-centered  Hardy--Littlewood maximal function $Mf$.
 If $f:X\to \R$ is a measurable function, then
\begin{equation}\label{e.max_funct_def}
Mf(x)=\sup_{B}\vint_{B} \lvert f(y)\rvert\,d\mu(y)\,,\qquad x\in X\,,
\end{equation}
where the supremum is taken over all balls $B\subset X$ such that $x\in B$. The sublinear operator $M$ is bounded on $L^s(X)$ for $1<s\leq \infty$, see \cite[Theorem 3.13]{MR2867756}.

  The first lemma
is a variant of a well known Euclidean estimate,
see~\cite[Lemma~3.1.1]{MR1411441}.  
 
\begin{lemma}\label{l.LocalRiesz_via_MaxFunct} 
  Let $\beta>0$ and let $f$ be   a nonnegative
measurable function on $X$.   Then  
\[
\int_{B(z,r)}\frac{f(y)d(z,y)^\beta}{\mu(B(z,d(z,y)))}\,d\mu(y)
\le C(c_\mu,\beta)r^\beta Mf(z)
\]
  for all $z\in X$ and all $r>0$.  
\end{lemma}
\begin{proof}
Define $r_i= 2^{-i} r$, $i=0,1,2,\dots$. 
  Using the doubling property of $\mu$
and the definition of the maximal function,
we   obtain
\begin{align*}
\int_{B(z, r)} \frac{f(y)d(z,y)^\beta}{\mu(B(z,d(z,y)))}\,d\mu(y)
&=\sum_{i=0}^{\infty} \int_{B(z,r_i)\setminus B(z,r_{i+1})} \frac{f(y)d(z,y)^\beta}{\mu(B(z,d(z,y)))}\,d\mu(y) \\
&\le \sum_{i=0}^{\infty} \int_{B(z,r_i)\setminus B(z,r_{i+1})} \frac{f(y)r_i^\beta}{\mu(B(z,r_{i+1}))}\,d\mu(y) \\
&\le  \sum_{i=0}^{\infty} \frac{r_i^\beta}{\mu(B(z,r_{i+1}))} \int_{B(z,r_i)} f(y) \, d\mu (y)\\
&\le  c_\mu \sum_{i=0}^{\infty}  r_i^\beta \vint_{B(z,r_i)} f(y) \, d\mu (y)\\
&\le c_\mu r^\beta Mf(z)\sum_{i=0}^{\infty} 2^{-i\beta} \ = \ C(c_\mu,\beta) r^\beta Mf(z). \qedhere
\end{align*}
\end{proof}

  The second lemma shows that the maximal function $Mf$
can be (essentially) used as a Haj{\l}asz $\beta$-gradient of $I_\beta f$.
The method of proof 
is standard, we refer to
\cite{MR1372029,MR3108871},
 and it    requires the validity of
the rather technical \emph{kernel estimate}~\eqref{e.kernelLemma}.
However, such estimate (or a suitable variant) holds, 
for instance, if the measure
is $Q$-uniform or Ahlfors $Q$-regular,
see Example~\ref{ex.unif_meas} and Remark~\ref{RemarkRieszRegularSpace}.  

\begin{lemma}\label{l.MaxFunctionIsHajlaszGradient}
  Let $0<\beta<\eta$ and let $f$ be  
a nonnegative measurable function
on $X$ such that $I_\beta f$ is finite everywhere in $X$.   
Assume that there exists $c_K>0$ such that
for all $w,y\in X$, $w\not=y$, we have
\begin{equation}\label{e.kernelLemma}
\left\lvert \frac{d(w,z)^\beta}{\mu(B(w,d(w,z)))}-\frac{d(y,z)^\beta}{\mu(B(y,d(y,z)))}   \right\rvert
\le c_K \:  \frac{d(w,y)^\eta }{d(w,z)^{\eta-\beta}\mu(B(w,d(w,z)))}
\end{equation}
for all $z\in X\setminus B(w,2d(w,y))$. Then there is a constant  $C_1=C(c_\mu,\beta,\eta,c_K)>0$  such that 
\begin{equation}\label{e.h_forall}
\lvert I_\beta f(w)-I_\beta f(y)\rvert
\le C_1d(w,y)^{\beta}\left(Mf(w)+Mf(y)\right)\,
\end{equation}
for every $w,y\in X$.
In particular, it follows that $C_1 Mf$ is a Haj{\l}asz $\beta$-gradient of $I_\beta f$. 
\end{lemma}

\begin{proof}
Let $w,y\in X$ with $w\not=y$. Then,
\begin{align*}
\lvert I_\beta f(w)-I_\beta f(y)\rvert &= 
\left\lvert \int_X \frac{f(z)d(w,z)^\beta}{\mu(B(w,d(w,z)))}\,d\mu(z)
-\int_X \frac{f(z)d(y,z)^\beta}{\mu(B(y,d(y,z)))}\,d\mu(z)\right\rvert\\&
\le \int_{B(w,2d(w,y))}\frac{f(z)d(w,z)^\beta}{\mu(B(w,d(w,z)))}\,d\mu(z)
+\int_{B(w,2d(w,y))}\frac{f(z)d(y,z)^\beta}{\mu(B(y,d(y,z)))}\,d\mu(z)\\
&\qquad + \int_{X\setminus B(w,2d(w,y))}
f(z)\left\lvert \frac{d(w,z)^\beta}{\mu(B(w,d(w,z)))}-\frac{d(y,z)^\beta}{\mu(B(y,d(y,z)))}   \right\rvert\,d\mu(z)\,.
\end{align*}

By Lemma \ref{l.LocalRiesz_via_MaxFunct}, we have
\[
\int_{B(w,2d(w,y))}\frac{f(z)d(w,z)^\beta}{\mu(B(w,d(w,z)))}\,d\mu(z)
\le C(c_\mu,\beta)d(w,y)^\beta Mf(w)
\]
and
\begin{align*}
\int_{B(w,2d(w,y))}\frac{f(z)d(y,z)^\beta}{\mu(B(y,d(y,z)))}\,d\mu(z)
&   \le \int_{B(y,3d(w,y))}\frac{f(z)d(y,z)^\beta}{\mu(B(y,d(y,z)))}\,d\mu(z)   \\
& \le C(c_\mu,\beta)d(w,y)^\beta Mf(y)\,.
\end{align*}
  On the other hand, the assumed kernel estimate
\eqref{e.kernelLemma} gives  
\begin{align*}
\int_{X\setminus B(w,2d(w,y))} &
f(z)\left\lvert \frac{d(w,z)^\beta}{\mu(B(w,d(w,z)))}-\frac{d(y,z)^\beta}{\mu(B(y,d(y,z)))}   \right\rvert \,d\mu(z)\\
& \le c_K d(w,y)^\eta\int_{X\setminus B(w,2d(w,y))}
 \frac{ f(z)d(w,z)^{\beta-\eta}}{\mu(B(w,d(w,z)))}\,d\mu(z)\,.
\end{align*}
Write $r_j=2^{j+1}d(w,y)$ for every $j=0,1,\ldots$. Since $\eta>\beta$, we have
\begin{align*}
\int_{X\setminus B(w,2d(w,y))}
 \frac{ f(z)d(w,z)^{\beta-\eta}}{\mu(B(w,d(w,z)))}\,d\mu(z)
 &\le \sum_{j=0}^\infty (r_j)^{\beta-\eta}\int_{B(w,r_{j+1})\setminus B(w,r_j)}
  \frac{ f(z)}{\mu(B(w,r_j))}\,d\mu(z)\\
  &\le c_\mu\sum_{j=0}^\infty (r_j)^{\beta-\eta}\vint_{B(w,r_{j+1})}
 f(z)\,d\mu(z)\\
& \le C(c_\mu,\beta,\eta)d(w,y)^{\beta-\eta} Mf(w)\,.
  \end{align*}
 By combining the estimates above, we obtain
\[
\lvert I_\beta f(w)-I_\beta f(y)\rvert
\le C(c_\mu,\beta,\eta,c_K)d(w,y)^{\beta}\left(Mf(w)+Mf(y)\right)\,
\]
for every $w,y\in X$.
That is, we have $C_1Mf\in\mathcal{D}^\beta_H(I_\beta f)$ with
$C_1=C(c_\mu,\beta,\eta,c_K)$.
\end{proof}

\begin{example}\label{ex.unif_meas}
Assume that $\mu$ is a $Q$-uniform measure  in $X$ for some $Q>0$, that is, there exists a constant $C_1>0$ such that 
\[
\mu(B(x,r))=C_1r^Q,\text{ for all }x\in X \text{ and    all   }r>0\,.
\]
  If  $0<\beta<\min\{Q,1\}$, then the kernel estimate \eqref{e.kernelLemma} holds with $\beta<\eta=1$.  
To show this, we proceed as in \cite[Lemma 4.2]{MR2047654}. By the mean-value theorem,
  for every $s,t>0$  
there exists $0<\theta<1$ such that
\begin{equation}\label{e.meanvt}
\lvert s^{\beta-Q}-t^{\beta-Q}\rvert
\le (Q-\beta)\lvert (1-\theta)s+\theta t\rvert^{\beta-Q-1}\lvert s-t\rvert\,.
\end{equation}
Fix $w,y\in X$, $w\not=y$, and  $z\in X\setminus B(w,2d(w,y))$. By  $Q$-uniformity of $\mu$ and inequality~\eqref{e.meanvt},
we have
\begin{align*}
\left\lvert \frac{d(w,z)^\beta}{\mu(B(w,d(w,z)))}-\frac{d(y,z)^\beta}{\mu(B(y,d(y,z)))}   \right\rvert
& = \frac{1}{C_1}\left\lvert d(w,z)^{\beta-Q}-d(y,z)^{\beta-Q}  \right\rvert
\\&\le C(C_1,Q,\beta)\frac{\lvert d(w,z)-d(y,z)\rvert}{d(w,z)^{Q-\beta+1}}
\\&\le C(C_1,Q,\beta) \frac{d(w,y)}{d(w,z)^{1-\beta}\mu(B(w,d(w,z)))}\,.
\end{align*}
As an example, the $n$-dimensional
Lebesgue measure in  $\R^n$ 
is  $n$-uniform. Uniform measures were first studied in \cite{MR890162}.
\end{example}

\begin{remark}\label{RemarkRieszRegularSpace}
  In an Ahlfors $Q$-regular  space $X$, with $Q>0$, 
the following version of the Riesz potential is often used: 
\[
\mathcal{I}_\beta f(x)=\int_X\frac{f(y)}{d(x,y)^{Q-\beta}}\,d\mu(y)\,,\quad x\in X\,,
\]
see \cite{MR2047654,MR1372029} and references therein. By the $Q$-regularity, 
  $\mathcal{I}_\beta f$ is %
pointwise comparable with the potential $I_\beta f$, that is,
there exists a constant $C>0$ such that 
\begin{equation}\label{eq.rieszes}
C^{-1}\mathcal{I}_\beta f(x)\le I_\beta f(x)\le C\mathcal{I}_\beta f(x)\,,
\end{equation}
for every $x\in X$.

  If $X$ is $Q$-regular and $0<\beta<\min\{Q,1\}$, then one can obtain results analogous  
to Lemma~\ref{l.LocalRiesz_via_MaxFunct} and Lemma~\ref{l.MaxFunctionIsHajlaszGradient} for $\mathcal{I}_\beta f$. 
  Observe   that the 
analogue of \eqref{e.kernelLemma} reads as
\[
\left\lvert \frac{1}{d(w,z)^{Q-\beta}}-\frac{1}{d(y,z)^{Q-\beta}}  \right\rvert
\le c_K \frac{d(w,y)^{\eta}}{d(w,z)^{Q-\beta+\eta}}\,, \quad w,\,y\in X,\ \ z\in X\setminus B(w,2d(w,y))\,.
\]
  This  automatically holds for $0<\beta<1=\eta$, see the reasoning in Example~\ref{ex.unif_meas} for a proof.  
\end{remark}

\section{Haj{\l}asz and Riesz capacities}\label{s.HRcap}

  Let $\Omega\subset X$ be a bounded open set and let   $F\subset\Omega$. 
One of the main objects in this paper is 
a capacity of $F$ relative to $\Omega$ defined via 
$\beta$-Haj\l asz gradients. 
  The following definition is from~\cite{CV2021}. In   the case $\beta=1$, such capacities were earlier considered
in \cite{MR2424909}.

\begin{definition}\label{d.varcap}
Let $1\le p<\infty$, $0<\beta\le 1$, and let $\Omega\subset X$  be a bounded open set.   
The variational Haj{\l}asz $(\beta,p)$-capacity of a  closed   subset $F\subset\Omega$ is 
\begin{equation*}\label{e.Hajlasz cap}
\cp_{\beta,p}(F,\Omega)=   \inf_u\inf_g   \int_X  g(x)^p\,d\mu(x)\,,
\end{equation*}
where the  infimums are taken  over all   $u\in \Lip_\beta(X)$,   with $u\ge 1$ in $F$, $u=0$ in $X\setminus \Omega$, and all $g\in\mathcal{D}_H^{\beta}(u)$.   If there are no such functions $u$, we set $\cp_{\beta,p}(F,\Omega)=\infty$.  
\end{definition}

\begin{remark}\label{r.varcap_properties}

 If  $u\in\Lip_\beta(X)$ and  $g\in\mathcal{D}^\beta_H(u)$, then $v=\max\{0,\min\{u,1\}\}$ is a $\beta$-H\"older function and $g\in\mathcal{D}^\beta_H(v)$. Hence, we may also assume $0\le u\le 1$ in Definition~\ref{d.varcap}.

\end{remark}

 The following lemma   shows that in connected spaces  
capacities of balls   are comparable   to suitably scaled measures of balls.

\begin{lemma}\label{l.Hequiv}
  Assume that the metric space $X$ is connected, and let   $1\le p<\infty$   and $0<\beta\le 1$.
Then there is a constant $C=C(c_\mu,p)>0$ such that
\begin{equation}\label{e.Hcomp}
C^{-1}  r^{-\beta p}\mu(B(x,r))\le \cp_{\beta,p}(\overline{B(x,r)},B(x,2r))
\le C r^{-\beta p}\mu(B(x,r))
\end{equation}
for all $x\in X$ and all  $0<r<(1/8)\diam(X)$. 
Moreover, the second inequality in \eqref{e.Hcomp} holds even if $X$ is not connected. 
\end{lemma}

\begin{proof}
The first inequality in \eqref{e.Hcomp} is a consequence of
\cite[Example 4.5]{CV2021}, where connectivity of $X$ is used.
In order to show the second inequality in \eqref{e.Hcomp}, we
let $x\in X$ and $0<r<(1/8)\diam(X)$. Define
\[
u(y)=\max\left\{0,1-r^{-\beta} \: {\dist(y,B(x,r))^\beta}\right\}\ ,\qquad y\in X\,.\]
Then $u\in \Lip_\beta(X)$, 
  with a constant $r^{-\beta}$, and therefore $g=r^{-\beta}\ch{B(x,2r)}\in \mathcal{D}^\beta_H(u)$
by \cite[Lemma~3.3]{CV2021}. Since  
$u=1$ in $\overline{B(x,r)}$ and
$u=0$ in $X\setminus B(x,2r)$,
  we conclude that  
\[
\cp_{\beta,p}(\overline{B(x,r)},B(x,2r))
\le \int_X g(x)^p\,d\mu(x)\le r^{-\beta p}\mu(B(x,2r))\le c_\mu\,r^{-\beta p}\mu(B(x,r))\,.\qedhere
\]
\end{proof}

 Another capacity   that   we use in this paper is the Riesz  
$(\beta,p)$-capacity. 
Such capacities   are well known in Euclidean spaces; 
we refer to \cite{MR1411441} and \cite{MR946438}, and references therein. In metric spaces, variants of Riesz
capacities have been studied for instance in \cite{Nuutinen2015TheRC}.

\begin{definition}\label{d.RieszCap}
Let  $F\subset X$, $\beta>0$ and $p\ge 1$. The Riesz  
$(\beta,p)$-capacity of $F$ is 
\[
R_{\beta,p}(F)=\inf \bigl\{ \lVert f\rVert_p^p\colon f\ge 0 \text{ and }I_\beta f\ge 1\text{ on }F \bigr\}\,,
\]
where $\lVert f\rVert_p =\lVert f\rVert_{L^p(X)}$ is the Lebesgue $p$-norm of $f$ on $X$.
\end{definition}

  See Lemma~\ref{c.balls} for  an analogue of Lemma~\ref{l.Hequiv} for the Riesz $(\beta,p)$-capacities.

\section{Comparability of Riesz and Haj{\l}asz capacities}\label{s.cap_are_comparable}

In this section we show that, under certain geometric hypotheses on the metric and  measure, 
the capacities defined in terms of Riesz potentials and Haj{\l}asz gradients are comparable.
  One direction of the comparison, given in Theorem~\ref{p.HaCap_geq_RieszCap},
holds under much weaker assumptions than the other one in Theorem~\ref{th.HCapLessRCap}.
In particular, in the latter the kernel estimate~\eqref{e.kernelLemma} is assumed, together with
a reverse doubling condition. Recall also that we assume throughout the paper that the measure $\mu$
is doubling, with a constant $c_\mu$. 

The following chaining lemma, 
which will be applied in the proof of Theorem~\ref{p.HaCap_geq_RieszCap}, is a 
straightforward modification of a result in \cite[p.~30]{MR1683160}; 
hence we omit the proof.

\begin{lemma}
\label{l.chain.estimate.improved}
  Assume that the metric space $X$ is connected. Then  
there exists a constant $M$ such that for all $y\in X$ and all  $0<\rho<R<(3/8)\diam (X)$, there exists $k= k(c_\mu,y,\rho,R) \in \N$ and balls $B_0,...,B_k$ 
  satisfying the following properties:  
\begin{itemize}
\item[(i)] $B_0\subset X\setminus B(y,R)$ and $B_k\subset B(y,\rho)$,
\item[(ii)] $M^{-1}\diam(B_i)\le d(y,B_i)\le M\diam(B_i)$ for all
$i=0,1,2,\ldots,k$,
\item[(iii)] there is a ball $R_i\subset B_i\cap B_{i+1}$ such
that $B_i\cup B_{i+1}\subset MR_i$ for 
all $i=0,1,2,\ldots, k-1$,
\item[(iv)] No point of $X$ belongs to more than
$M$ balls $B_i$.
\end{itemize}
\end{lemma}

The first main result of this section  
gives an upper bound for the Riesz 
capacity in terms of the Haj{\l}asz capacity in connected metric spaces.

\begin{theorem}\label{p.HaCap_geq_RieszCap}
Assume that $X$ is connected and 
let $1\le p<\infty$ and $0<\beta\le 1$.
Moreover, let  
$E\subset X$ be a closed set, $x\in E$ and $0<r<(1/8)\diam(X)$.  Then
\begin{equation}\label{e.HaCap_geq_RieszCap}
\cp_{\beta,p} \big( E\cap \overline{B(x,r)},B(x,2r) \big)\ge C(c_\mu,p)R_{\beta,p}(E\cap \overline{B(x,r)})\,.
\end{equation}
\end{theorem}

\begin{proof}
Let $u\in \Lip_\beta(X)$ be such that $0\le u\le 1$, $u=1$ in $E\cap \overline{B(x,r)}$ and $u=0$ outside of $B(x,2r)$.
Denote by $\kappa$ the $\beta$-H\"older constant of $u$ in $X$ and let  $g\in\mathcal{D}_H^{\beta}(u)$. 
Theorem \ref{t.pp_poincare} implies that the
Haj{\l}asz $(\beta,1,1)$-Poincar\'e inequality
\begin{equation}\label{1-1poincare}
\vint_B\lvert u-u_B\rvert\,d\mu
\le 2\diam(B)^\beta \vint_{B} g\,d\mu
\end{equation} 
holds for all balls $B\subset X$. 
We claim that \eqref{1-1poincare}, together with the facts that $u=0$ outside of $B(x,2r)$ and $r< (1/8)\diam(X)$, give the following inequality
\begin{equation}\label{e.repres}
\lvert u(y)\rvert\le C(c_\mu)I_\beta g(y)\,,
\end{equation}
for every $y\in \overline{B(x,r)}$.  
We postpone the proof of \eqref{e.repres}
and finish the proof of the theorem while assuming \eqref{e.repres}.
It follows from the properties of $u$ and \eqref{e.repres} that
$I_\beta (C(c_\mu) g)\ge 1$ in $E\cap  \overline{B(x,r)}$. Thus
\[
C(c_\mu,p)\int_X g^p\,d\mu=\lVert C(c_\mu) g\rVert_p^p \ge R_{\beta,p}(E\cap \overline{B(x,r)})\,.
\]
Taking infimum over all $u$ and $g\in\mathcal{D}_H^{\beta}(u)$, we get
\begin{equation}
\cp_{\beta,p} \big( E\cap \overline{B(x,r)},B(x,2r) \big)\ge C(c_\mu,p)R_{\beta,p}(E\cap \overline{B(x,r)})\,.
\end{equation}

For convenience of the reader, we 
give the proof of \eqref{e.repres}, which is contained in \cite{MR2569546}.
Let $y\in \overline{B(x,r)}$, $R=3r< (3/8)\diam(X)$ and $0<\rho<R$,  
  and let $B_0, B_1, \ldots,B_k$ be the corresponding family of balls given by Lemma~\ref{l.chain.estimate.improved}.
  
Notice that since $B_0\subset X\setminus B(y,3r) \subset  X\setminus B(x,2r)$,  we have  $u=0$ on $B_0$. 
Then, we can write
\begin{align*}
\lvert u(y)\rvert=\lvert u(y)-u_{B_0}\rvert
&\le \sum_{i=0}^{k-1} \lvert u_{B_{i+1}}-u_{B_i}\rvert + 
\lvert u(y)-u_{B_k}\rvert\\
&\le \sum_{i=0}^{k-1} \left(\lvert u_{B_{i+1}}-u_{R_i}\rvert + 
\lvert u_{B_i}-u_{R_i}\rvert\right) + 
\kappa \rho^\beta\\
&\le C(c_\mu,M)\sum_{i=0}^k \vint_{B_i} \lvert u-u_{B_i}\rvert\,d\mu + \kappa \rho^\beta\,.
\end{align*}
Applying the Poincar\'e inequality \eqref{1-1poincare}, we get 
\begin{align*}
\lvert u(y)\rvert
&\le C(c_\mu,M)\sum_{i=0}^k \diam(B_i)^\beta \vint_{B_i} g\,d\mu
+ \kappa \rho^\beta\\
&= C(c_\mu,M)\sum_{i=0}^k  \int_{B_i} \frac{g(z) \diam(B_i)^\beta}{\mu(B_i)}\,d\mu(z)
+ \kappa \rho^\beta\\
&\le C(c_\mu,M)\sum_{i=0}^k  \int_{B_i} 
\frac{g(z)d(y,z)^\beta}{\mu(B(y,d(y,z)))}\,d\mu(z)
+ \kappa \rho^\beta\\
&\le C(c_\mu,M)I_\beta g(y)+\kappa \rho^\beta\,,
\end{align*}
  and inequality \eqref{e.repres} follows
by letting $\rho\to 0$.  
\end{proof}

Under additional assumptions on $X$ inequality \eqref{e.HaCap_geq_RieszCap} can be reversed,
  see Theorem~\ref{th.HCapLessRCap}.   Before a precise formulation of this   reverse estimate,  
we prove an auxiliary lemma.

\begin{lemma}\label{CommonEstimate}
Let  $1\le p<\infty$  and $0<\beta\le 1$, and 
  assume that %
$\mu$ satisfies the quantitative reverse doubling condition \eqref{e.reverse_doubling}
for some exponent $\sigma>\beta p$. 
  Let   $x\in X$,  $0<r<(1/8)\diam(X)$  and   let $f\in L^p(X)$ be   a nonnegative function such that
\begin{equation}\label{e_Integral}
\int_{X\setminus B(z,r/5)} \frac{f(y)d(z,y)^\beta}{\mu(B(z,d(z,y)))}\,d\mu(y)\geq \frac{1}{2}
\end{equation}
for some $z\in \overline{B(x,r)}$. Then  
there exists a constant $C=C(c_\mu,c_\sigma,\beta,\sigma,p)$ such that 
\[\|f\|^p_{L^p(X)}\ge Cr^{-\beta p}\mu(B(x,r))\,.\]
\end{lemma}

\begin{proof}
Throughout the proof, we assume
that $X$ is bounded. The case when $X$ is unbounded
is treated in a similar way,
  but the essential difference is that then the sums below will have infinitely many terms.
Write   $r_k=2^kr/5$
for every $k=0,1,2,\ldots,K+1$, where $K\in\N$ is
chosen such that $r_{K} \leq \diam (X)< 2r_{K}$. Note that $X=B(z,r_{K+1})$
  and recall   that $z\in \overline{B(x,r)}$. Hence, 
by the doubling condition for the measure $\mu$ and by \eqref{e_Integral}, we have
\begin{align*}
\frac{1}{2}\mu(B(x,r))^\frac{1}{p}&\le \frac{c_\mu^4}{2} \mu(B(z,r/5))^{\frac{1}{p}}\\
&\le c_\mu^4 \mu(B(z,r/5))^{\frac{1}{p}}\int_{X\setminus B(z,r/5)} \frac{f(y)d(z,y)^\beta}{\mu(B(z,d(z,y)))}\,d\mu(y)\\
&= c_\mu^4 \sum_{k=0}^{K} \mu(B(z,r/5))^{\frac{1}{p}}\int_{B(z,r_{k+1})\setminus B(z,r_{k})} \frac{f(y)d(z,y)^\beta}{\mu(B(z,d(z,y)))}\,d\mu(y)\\
&\le c_\mu^4 \sum_{k=0}^{K} \mu(B(z,r/5))^{\frac{1}{p}} (r_{k+1})^\beta\int_{B(z,r_{k+1})\setminus B(z,r_{k})} \frac{f(y)}{\mu(B(z,r_k))}\,d\mu(y)\\
&\le c_\mu^5 \sum_{k=0}^{K} \mu(B(z,r/5))^{\frac{1}{p}} (r_{k+1})^\beta\vint_{B(z,r_{k+1})} f(y)\,d\mu(y)\\
&\le c_\mu^5 \sum_{k=0}^{K} \frac{\mu(B(z,r/5))^{\frac{1}{p}}}{\mu(B(z,r_{k+1}))^{\frac{1}{p}}} (r_{k+1})^\beta\biggl(\int_{B(z,r_{k+1})} f(y)^p\,d\mu(y)\biggr)^{\frac{1}{p}}\,.
\end{align*}
By the reverse doubling condition \eqref{e.reverse_doubling}, for every
$k=0,1,2\ldots,K$,  we get
\begin{align*}
\frac{\mu(B(z,r/5))^{\frac{1}{p}}}{\mu(B(z,r_{k+1}))^{\frac{1}{p}}} (r_{k+1})^\beta
\ =\frac{\mu(B(z,r/5))^{\frac{1}{p}}}{\mu(B(z,r_{k+1}))^{\frac{1}{p}}} (r_{k+1})^{\frac{\sigma}{p}}
(r_{k+1})^{\beta-\frac{\sigma}{p}}\le c_\sigma (r/5)^{\frac{\sigma}{p}}(r_{k+1})^{\beta-\frac{\sigma}{p}}\,,
\end{align*}
 since $r_{k+1}\leq 2 \diam (X)$. Hence, we can continue to estimate as follows
\begin{align*}
 \frac{1}{2}\mu(B(x,r))^\frac{1}{p} &\le  c_\mu^5 \sum_{k=0}^{K} \frac{\mu(B(z,r/5))^{\frac{1}{p}}}{\mu(B(z,r_{k+1}))^{\frac{1}{p}}} (r_{k+1})^\beta\biggl(\int_{B(z,r_{k+1})} f(y)^p\,d\mu(y)\biggr)^{\frac{1}{p}}\\
&\le 
 c_\mu^5  c_\sigma \sum_{k=0}^{K}  (r/5)^{\frac{\sigma}{p}}(r_{k+1})^{\beta-\frac{\sigma}{p}}\biggl(\int_{B(z,r_{k+1})} f(y)^p\,d\mu(y)\biggr)^{\frac{1}{p}}\\
&\le c_\mu^5 c_\sigma  (r/5)^{\beta}  \biggl(\int_{X} f(y)^p\,d\mu(y)\biggr)^{\frac{1}{p}}\ \sum_{k=0}^\infty 2^{(\beta-\frac{\sigma}{p})(k+1)}\,.
\end{align*}
Since $\sigma>\beta p$, the geometric series converges. By combining the above estimates, we get
\begin{equation*}%
\mu(B(x,r))^\frac{1}{p}\le 
C(c_\mu,c_\sigma,\beta,\sigma,p)r^{\beta}  \biggl(\int_{X} f(y)^p\,d\mu(y)\biggl)^{\frac{1}{p}}\,.
\end{equation*}
The estimate
$\|f\|^p_{L^p(X)}\ge Cr^{-\beta p}\mu(B(x,r))$
follows after  simplication. 
\end{proof}

  The second main result of this section is a converse of Theorem~\ref{p.HaCap_geq_RieszCap}.
Observe that assumption \eqref{e.kernel} is the same as the kernel estimate~\eqref{e.kernelLemma}
that appears in Lemma~\ref{l.MaxFunctionIsHajlaszGradient}.  

\begin{theorem}\label{th.HCapLessRCap}
  Assume that the metric space $X$ is complete. Let   $1<p<\infty$ and $0<\beta\le 1$, 
  and assume that $\mu$   satisfies the quantitative reverse doubling condition \eqref{e.reverse_doubling}
for some exponent $\sigma>\beta p$.   In addition,  
assume that there exist $\eta>\beta$ and $c_K>0$ such that
for all $w,y\in X$, $w\not=y$, we have
\begin{equation}\label{e.kernel}
\left\lvert \frac{d(w,z)^\beta}{\mu(B(w,d(w,z)))}-\frac{d(y,z)^\beta}{\mu(B(y,d(y,z)))}   \right\rvert
\le c_K \:  \frac{d(w,y)^\eta }{d(w,z)^{\eta-\beta}\mu(B(w,d(w,z)))}
\end{equation}
for all $z\in X\setminus B(w,2d(w,y))$.
Then there exists a constant $C=C(c_K,c_\mu,c_\sigma,\beta,\eta,\sigma,p)$ such that
\begin{equation}\label{e.goal}
 \cp_{\beta,p} \big( E\cap \overline{B(x,r)},B(x,2r) \big)\le C\, R_{\beta,p}(E\cap \overline{B(x,r)})
\end{equation}
  whenever $E\subset X$ is a closed set, $x\in E$ and $0<r<(1/8)\diam(X)$.   
\end{theorem}

\begin{proof}
We may assume that the Riesz $(\beta,p)$-capacity   on  
the right-hand side of \eqref{e.goal} is finite.
Fix a nonnegative function $f\in L^p(X)$ such that $I_\beta f\ge 1$ on $F=E\cap \overline{B(x,r)}$. 
It suffices to   show   %
that
\begin{equation}\label{e.r_desired}
\cp_{\beta,p}(F,B(x,2r))\le
C(c_K,c_\mu,c_\sigma,\beta,\eta,\sigma,p)\int_{X} f(y)^p\,d\mu(y)
\end{equation}
  since then the claim~\eqref{e.goal} follows 
by taking infimum over all $f$ as above. %
In order to prove inequality \eqref{e.r_desired}, we consider several cases.

First we assume that 
\begin{equation}\label{e.riesz-first}
\int_{X\setminus B(z,r/5)} \frac{f(y)d(z,y)^\beta}{\mu(B(z,d(z,y)))}\,d\mu(y)\geq \frac{1}{2}\,
\end{equation}
for some $z\in F \subset \overline{B(x,r)} $ . Then,   by Lemma~\ref{CommonEstimate}, we have
\[
r^{-\beta p}\mu(B(x,r))\le 
C(c_\mu,c_\sigma,\beta,\sigma,p)\int_{X} f(y)^p\,d\mu(y).
\]
By monotonicity and Lemma \ref{l.Hequiv}, we get
\begin{equation}\label{e.first_case_below}
\cp_{\beta,p}(F,B(x,2r))\le \cp_{\beta,p}(\overline{B(x,r)},B(x,2r))\le C(c_\mu,c_\sigma,\beta,\sigma,p)\int_{X} f(y)^p\,d\mu(y)\,,
\end{equation}
  and so \eqref{e.r_desired} holds in this case.

Next we  consider the case when  \eqref{e.riesz-first}
does not hold for any   $z\in F$. Since  $I_\beta f\ge 1$ on $F$, we have 
\begin{equation}\label{e.contra}
\int_{B(z,r/5)} \frac{f(y)d(z,y)^\beta}{\mu(B(z,d(z,y)))}\,d\mu(y)
= I_\beta f(z)-\int_{X\setminus B(z,r/5)} \frac{f(y)d(z,y)^\beta}{\mu(B(z,d(z,y)))}\,d\mu(y) 
\geq \frac{1}{2}\,,
\end{equation}
for every $z\in F$.

Assume first that $f$ is bounded in $X$. Let
\[
\psi(z)=\max \bigl\{0,1-r^{-\beta}d(z,B(x,r))^\beta\bigr\}\,
\]
for every $z\in X$. 
Then $0\le\psi\le 1$, $\psi=0$ in $X\setminus B(x,2r)$, $\psi=1$ in $\overline{B(x,r)}$,
and  $\psi$ is a $\beta$-H\"older function in $X$ with a constant $r^{-\beta}$.
Define $h=2f\ch{B(x,2r)}$  and
\begin{equation}\label{TestFunctionRieszCapacity}
v(z)=I_\beta h(z)\psi(z)\,,\qquad z\in X\,.
\end{equation}
By inequality  \eqref{e.contra}, for every $z\in F=E\cap \overline{B(x,r)}$   we have  
\[
1\leq \int_{X} \frac{2f(y)\ch{B(x,2r)}(y)d(z,y)^\beta}{\mu(B(z,d(z,y)))}\,d\mu(y)
= I_\beta\bigl(2f\ch{B(x,2r)}\bigr)(z)=I_\beta h(z).
\]
  Thus  
$v\ge 1$ in $F=E\cap \overline{B(x,r)}$  and $v=0$ in $X\setminus B(x,2r)$. 
By using Lemma \ref{l.LocalRiesz_via_MaxFunct} and properties of $h$, we see that $I_\beta h$ is finite everywhere in $X$. 
Since $Mh$ is bounded in $X$, Lemma \ref{l.MaxFunctionIsHajlaszGradient} implies
that $I_\beta h\in \Lip_\beta(X)$ 
and   $C_1Mh\in\mathcal{D}_H^\beta(I_\beta h)$,
where   $C_1=C(c_\mu,\beta,\eta,c_K)>0$. 

Let $M = \sup_{z \in B(x,2r)}I_\beta h(z)$ and $u =  \min\{I_\beta h,M\}$. Then  $u$ is a bounded $\beta$-H\"older function that coincides with $I_\beta h$ on $B(x,2r)$; thus $v = u\psi$ in $X$. 
We also have that $C_1Mh\in\mathcal{D}_H^\beta(u)$. Lemma~\ref{l.Leibniz} implies that  the function $v$ is $\beta$-H\"older in $X$ and   
\[
g_v= \left(C_1Mh\lVert \psi\rVert_\infty + r^{-\beta} I_\beta h\right)\ch{\{\psi\not=0\}} =\left(C_1Mh\lVert \psi\rVert_\infty + r^{-\beta}  u \right)\ch{\{\psi\not=0\}} \in\mathcal{D}_H^{\beta}(v)\,.
\]
  Thus   %
the pair $v$ and $g_v$ is admissible for testing the Hajlasz $(\beta,p)$-capacity  of $F$ relative to $B(x,2r)$  and
\begin{align*}
\cp_{\beta,p}(F,B(x,2r))
&\le\int_X  g_v(z)^p\,d\mu(z)\\&\le 
2^p  C_1^p\int_{B(x,2r)} (Mh(z))^p\,d\mu(z) + 2^pr^{-\beta p}\int_{B(x,2r)} (I_\beta h(z))^p\,d\mu(z).\end{align*}
Since $h\le 2f$, we have $Mh\le 2 Mf$.  By Lemma \ref{l.LocalRiesz_via_MaxFunct}, we also have 
\begin{align*}
I_\beta h(z)
&= 2\int_{B(x,2r)} \frac{f(y)d(z,y)^\beta}{\mu(B(z,d(z,y)))}\,d\mu(y)
\\ &\le 2\int_{B(z,4r)} \frac{f(y)d(z,y)^\beta}{\mu(B(z,d(z,y)))}\,d\mu(y)
\le C(c_\mu,\beta)r^\beta Mf(z)
\end{align*}
for every $z\in B(x,2r)$.
Hence, 
we see that
\begin{align*}
\cp_{\beta,p}(F,B(x,2r))
\le C(C_1,p,c_\mu,\beta)\int_X (Mf(z))^p\,d\mu(z)\le C(C_1,p,c_\mu,\beta)\int_X f(z)^p\,d\mu(z)\,.
\end{align*}
 The last step follows
from the Hardy--Littlewood   maximal   theorem since $p>1$, see 
\cite[Theorem~3.13]{MR2867756}. We have shown
that inequality \eqref{e.r_desired} holds.

In case $f$ is not bounded  in $X$ and \eqref{e.contra} holds for all $z\in F$,  we write $f_k(y)=\min\{k,f(y)\}$ for every $y\in X$ and $k\in\N$. Observe that $f_k\to f$ almost everywhere in $X$, as $k\to \infty$.
Hence, by Fatou's lemma
\begin{align*}
\frac{1}{2}&\leq \int_{B(z,r/5)} \frac{f(y)d(z,y)^\beta}{\mu(B(z,d(z,y)))}\,d\mu(y)
\le \liminf_{k\to\infty} \int_{B(z,r/5)} \frac{f_k(y)d(z,y)^\beta}{\mu(B(z,d(z,y)))}\,d\mu(y)
\end{align*}
for every $z\in F$. We claim that there exists $k\in \N$ such that
\begin{equation}\label{e.option}
\frac{1}{3c_\mu}\le \int_{B(z,r/5)} \frac{f_k(y)d(z,y)^\beta}{\mu(B(z,d(z,y)))}\,d\mu(y)
\end{equation}
for all $z\in F$.
Assume that such a  number  $k$ does not exist. Then for every $k\in \N$ there exists
$z_k\in F$ such that 
\[
\frac{1}{3c_\mu}> \int_{B(z_k,r/5)} \frac{f_k(y)d(z_k,y)^\beta}{\mu(B(z_k,d(z_k,y)))}\,d\mu(y)\,.
\]
Since $F\subset X$ is a closed and bounded set and $X$ is a complete doubling space, 
by \cite[Proposition~3.1]{MR2867756}
we find that $F$ is compact. In particular, there exists
a subsequence $(z_{k_m})_{m\in\N}$ of $(z_k)_{k\in \N}$ such that
$\lim_{m\to\infty} z_{k_m}= z_0\in F$. By Fatou's lemma
\begin{align*}
\frac{1}{3c_\mu}\ge \liminf_{m\to\infty}\int_{B(z_{k_m},r/5)} \frac{f_{k_m}(y)d(z_{k_m},y)^\beta}{\mu(B(z_{k_m},d(z_{k_m},y)))}\,d\mu(y)\ge 
\frac{1}{c_\mu} \int_{B(z_{0},r/5)} \frac{f(y)d(z_{0},y)^\beta}{\mu(B(z_{0},d(z_{0},y)))}\,d\mu(y)\,.
\end{align*}
This is a contradiction with \eqref{e.contra},  since $z_0\in F$. 
We have shown that there exists $k\in\N$ such that \eqref{e.option} holds for all $z\in F$.
Using a test function $v$ as in \eqref{TestFunctionRieszCapacity} but with
$h=3c_\mu f_k\ch{B(x,2r)}$, we repeat the reasoning above and finish the proof.
\end{proof}

Combining   Theorems~\ref{p.HaCap_geq_RieszCap} and~\ref{th.HCapLessRCap},   we obtain the following result 
on   the   comparability of the Riesz   and   Haj{\l}asz capacities.

\begin{corollary}\label{c.HCapEqvRCap}
  Assume that the metric space $X$ is complete and connected. Let $1<p<\infty$ and $0<\beta\le 1$,
and assume that $\mu$  
satisfies the quantitative reverse doubling condition \eqref{e.reverse_doubling}
for some exponent $\sigma>\beta p$ and that the kernel
estimate \eqref{e.kernel} holds  for some $\eta>\beta$ and $c_K>0$. 
Then there is a constant 
 $C=C(c_K,c_\mu,c_\sigma,\beta,\eta,\sigma,p)$  such that
\begin{equation}\label{e.comparabilityRieszHajlasz}
 C^{-1} \: R_{\beta,p}(E\cap \overline{B(x,r)}) \leq \cp_{\beta,p} \big( E\cap \overline{B(x,r)},B(x,2r) \big)\le C\: R_{\beta,p}(E\cap \overline{B(x,r)})\,,
\end{equation}
  whenever $E\subset X$ is a closed set,   $x\in E$ and $0<r<(1/8)\diam(X)$. 
\end{corollary}

\begin{remark}\label{r.Qreg}
In an Ahlfors $Q$-regular space $X$ it is natural to consider capacities defined in terms of Riesz potentials $\mathcal{I}_\beta f$, see Remark~\ref{RemarkRieszRegularSpace}. For these capacities, one can obtain the results analogous to Theorem~\ref{p.HaCap_geq_RieszCap} and Theorem~\ref{th.HCapLessRCap}. The proofs  are  similar and use the  accordingly  modified statements of Lemma \ref{l.LocalRiesz_via_MaxFunct}, Lemma \ref{l.MaxFunctionIsHajlaszGradient} and Lemma~\ref{CommonEstimate}. 
In this setting  with $0<\beta<1$,   the kernel estimate \eqref{e.kernel} is not required anymore, see Remark~\ref{RemarkRieszRegularSpace}, 
  and the quantitative reverse doubling condition \eqref{e.reverse_doubling} holds with $\sigma=Q$.
Hence, it follows that the corresponding  Riesz $(\beta,p)$-capacity   $\mathcal{I}_\beta f$ 
(and by~\eqref{eq.rieszes} also $I_\beta f$)   
is comparable to the   Haj{\l}asz $(\beta,p)$-capacity in any 
  complete and connected   Ahlfors $Q$-regular space with  $0<\beta<1$   and   $Q>\beta p$.
  This proves Theorem~\ref{thm.RH_intro}.  
\end{remark}

\section{Hausdorff content density condition}\label{s.aux_restimate}

In this section, we introduce a version of the Hausdorff content density condition, 
slightly different from the one that was used in~\cite{MR3631460}.
This is an auxiliary condition that 
  will be used in Section~\ref{s.lequiv_density} to connect 
Riesz and Haj{\l}asz   capacity density conditions.  

\begin{definition}
 The ($\rho$-restricted) Hausdorff content of codimension $q\ge 0$ 
of a set $F\subset X$ is defined by  
\[
\Ha^{\mu,q}_\rho(F)=\inf\Biggl\{\sum_{k} \mu(B(x_k,r_k))\,r_k^{-q} :
F\subset\bigcup_{k} B(x_k,r_k)\text{ and } 0<r_k\leq \rho  \Biggr\}.
\]
A closed set $E\subset X$ satisfies the Hausdorff content density condition of codimension $q\ge 0$ if  there is a constant $c_0>0$ such that
\begin{equation}\label{e.hausdorff_content_density}
\Ha^{\mu,q}_r(E\cap \overline{B(x,r)})\ge c_0 \: \Ha^{\mu,q}_r(\overline{B(x,r)})
\end{equation}
for   all   $x\in E$ and all $0<r<\diam(E)$.
\end{definition}

\begin{remark}\label{r.Hlasku}
Let $q\ge 0$.
Then it is easy to show that
$r^{-q}\mu(B(x,r))=\Ha^{\mu,q}_r(B(x,r))$ and that
\[
r^{-q}\mu(B(x,r))\le \Ha^{\mu,q}_r(\overline{B(x,r)}) \le C(c_\mu)\, r^{-q}\mu(B(x,r))
\] for   all   $x\in X$ and all $r>0$.
 This can be seen as an analogue of Lemma \ref{l.Hequiv}
for the Hausdorff content. 
In particular, it follows that 
a closed set $E\subset X$ satisfies the Hausdorff content density condition of codimension $q\ge 0$ if, and only if
 there is a constant $c_1>0$ such that 
\[
\Ha^{\mu,q}_r(E\cap \overline{B(x,r)})\ge c_1 r^{-q}\mu(B(x,r))  
\]
for every $x\in E$ and all  $0<r<\diam(E)$.
 This equivalent condition   was taken as the   definition
of the Hausdorff content density condition
of codimension $q$ in \cite{CV2021}. 

\end{remark}

By definition, we have  $\Ha^{\mu,q}_{2\diam(X)}(F)\le \Ha^{\mu,q}_r(F)$
if $0<r\le 2\diam(X)$ and $F\subset X$. Theorem~\ref{t.improve_H} gives  a  condition  under   which this inequality can be reversed.

\begin{theorem}\label{t.improve_H}
  Let %
$q>0$.  
The following conditions are 
equivalent:
\begin{enumerate}
\item[\textup{(i)}] There is a constant $c_q>0$ such that the quantitative
reverse doubling condition
\begin{equation}\label{e.reverse_doubling_q}
 \frac{\mu(B(x,r))}{\mu(B(x,R))} \le c_q\Bigl(\frac{r}{R}\Bigr)^q
\end{equation}
   holds  
for   all $x\in X$ and %
all   $0<r<R\le 2\diam(X)$.
\item[\textup{(ii)}] There is a constant $C_2>0$ such that 
\[\Ha^{\mu,q}_r(F)\le C_2\Ha^{\mu,q}_{2\diam(X)}(F)\]
  whenever   
$F\subset \overline{B(x,r)}$, with $x\in X$ and $r>0$.
\item[\textup{(iii)}] There is a constant $C_3>0$ such that 
\[
r^{-q}\mu(B(x,r))  \le C_3\Ha^{\mu,q}_{2\diam(X)}(B(x,r))\]
for   all $x\in X$ and all   $r>0$.
\end{enumerate}
\end{theorem}

\begin{proof}
We assume that condition (i) holds and we prove condition (ii). For this purpose, we fix $F\subset \overline{B(x,r)}$ with $x\in X$ and $r>0$.
Let $\{B(x_k,r_k)\}_{k\in I}$ be a finite or countable cover of $F$ such that
$B(x_k,r_k)\cap F\not=\emptyset$ and $0< r_k\le 2\diam(X)$ for all $k\in I$.
If $r_k\le r$ for all $k\in I$, then
\[
\Ha^{\mu,q}_r(F)\le \sum_{k\in I} \mu(B(x_k,r_k))r_k^{-q}\,.
\]
Next we assume that 
at least one of the radii $r_k$, $k\in I$, is not bounded by $r$. Without loss of generality, we may assume that
$r_1> r$.
By the $5r$-covering lemma \cite[Lemma 1.7]{MR2867756}, we obtain $N$ 
  pairwise   disjoint balls $B(y_i,r/5)$ such that
$F\subset \bigcup_{i=1}^N B(y_i,r)$ and $y_i\in F$ for every $i=1,\ldots,N$.
The doubling condition implies that $N\le C(c_\mu)$.
By  the  assumed reverse doubling condition, followed by the doubling condition,  we get
\begin{align*}
\Ha^{\mu,q}_r(F)&\le\sum_{i=1}^N \mu(B(y_i,r))r^{-q}
\le c_q \sum_{i=1}^N \mu(B(y_i,r_1))r_1^{-q}
\\& \le Nc_\mu^2 c_q 
\mu(B(x_1,r_1)) r_1^{-q}\le C(c_q,c_\mu)\sum_{k\in I} \mu(B(x_k,r_k))r_k^{-q}\,.
\end{align*}
In any case, we see that
\[
\Ha^{\mu,q}_r(F)\le C(c_q,c_\mu)\sum_{k\in I} \mu(B(x_k,r_k))r_k^{-q}\,,
\]
and condition (ii) with a constant $C_2=C(c_q,c_\mu)>0$ follows by taking infimum over all covers
$\{B(x_k,r_k)\}_{k\in I}$ of $F$ as above.

 Condition  (ii) implies condition (iii) by choosing $F=B(x,r)$  and using Remark \ref{r.Hlasku}. 
  Finally,   assume that
condition (iii) holds. Let $x\in X$ and $0<r<R\le 2\diam(X)$. Then 
Remark \ref{r.Hlasku} implies that \begin{align*}
r^{-q}\mu(B(x,r))&\le C_3\Ha^{\mu,q}_{2\diam(X)}(B(x,r))  \le C_3\Ha^{\mu,q}_{2\diam(X)}(B(x,R))
\\&\le C_3\Ha^{\mu,q}_R(B(x,R))=C_3R^{-q}\mu(B(x,R))\,.
\end{align*}
Thus, condition (i) holds with a constant $c_q=C_3$.
\end{proof}

Assuming \eqref{e.reverse_doubling_q}  for some $q>0$,   then a closed set $E\subset X$ satisfies the Hausdorff content density condition of codimension $q\ge 0$ if and only if
\[
\Ha^{\mu,q}_{2\diam(X)}(E\cap \overline{B(x,r)})\ge C \: \Ha^{\mu,q}_{2\diam(X)}(\overline{B(x,r)})
\]
for   all   $x\in E$ and all  $0<r<\diam(E)$.
This is a direct consequence of Theorem \ref{t.improve_H}.

The next lemma is an extension of \cite[Proposition 3.12]{MR3277052}
to metric spaces, with a new proof;  see
also \cite[Corollary 5.1.14]{MR1411441}.
 
  It 
gives a Riesz capacity estimate for sets satisfying the Hausdorff content density condition.
This result and also its consequence in Lemma~\ref{c.balls}
will be applied in the proof of Theorem~\ref{t.Riesz_and_Hajlasz},  
which relates the density conditions corresponding to the Riesz capacity, 
the Haj{\l}asz capacity and the Hausdorff content.

\begin{lemma}\label{l.hajlasz implies capacity}
  Let $1<p<\infty$ and $0<\beta\le 1$,
and assume that $\mu$  
satisfies the quantitative reverse doubling condition \eqref{e.reverse_doubling}
for some exponent $\sigma>\beta p$. 
  Moreover, let   $E\subset X$ be a closed set
that
satisfies the Hausdorff content density condition \eqref{e.hausdorff_content_density} for some $0 < q < \beta p$.
Then there is a constant  $C=C(c_1,c_\mu,c_\sigma,\sigma,\beta,q,p)>0$  such that
\[
R_{\beta,p}(E\cap \overline{B(x,r)})\ge Cr^{-\beta p}\mu(B(x,r)) 
\]
for all $x\in E$ and all $0<r<(1/8)\diam(E)$. 
Here the constant $c_1$ is defined in \eqref{e.hdensity}.  
\end{lemma}

\begin{proof}
 By Remark \ref{r.Hlasku},  there is a constant  $c_1>0$  such that
\begin{equation}\label{e.hdensity}
\Ha^{\mu, q}_{r}(E\cap \overline{B(x,r)})\geq c_1 r^{-q}\mu(B(x,r))\,,
\end{equation}
for   all   $x\in E$ and   all   $0<r<\diam (E)$.

Fix $x\in E$ and $0<r<(1/8)\diam(E)$, and write $F=E\cap \overline{B(x,r)}$.  
 We may assume that $R_{\beta,p}(F)<\infty$. 
  Let $f\in L^p(X)$ be   a nonnegative  test function for the capacity $R_{\beta,p}(F)$. Then,  in particular, 
\begin{equation}\label{e.X_is_large}
1\le I_\beta f(z)=\int_{X}\frac{f(y)d(z,y)^\beta}{\mu(B(z,d(z,y)))}\,d\mu(y)
\end{equation}
for every $z\in  F$.
We consider two cases. 

First,   assume   that
\[
\frac{1}{2}\leq \int_{B(z,r/5)} \frac{f(y)d(z,y)^\beta}{\mu(B(z,d(z,y)))}\,d\mu(y)
\]
for every $z\in F$. In this case, 
 we proceed as in the proof
of \cite[Theorem 4.3]{Nuutinen2015TheRC}. 
Namely, we fix $z\in F$ and 
define $r_i=2^{-i}r/5$, $i=0,1,2,\dots$.  Estimating as in Lemma \ref{l.LocalRiesz_via_MaxFunct}, we obtain
\begin{align*}
1&\leq 2 \int_{B(z,r/5)} \frac{f(y)d(z,y)^\beta}{\mu(B(z,d(z,y)))}\,d\mu(y)\le 2c_\mu \sum_{i=0}^{\infty} r_i^\beta \vint_{B(z,r_i)} f(y) \, d\mu (y)\\
&\le  2c_\mu \sum_{i=0}^{\infty} \frac{r_i^\beta}{\mu(B(z,r_{i}))^{1/p}}\left(\int_{B(z,r_i)} f(y)^p \, d\mu (y)\right)^{1/p}.
\end{align*}
For every   $\delta>0$, there is a constant $C(\delta)>0$ such that 
\[1=C(\delta)\sum_{i=0}^{\infty}2^{-i \delta}=C(\delta)5^{\delta}\sum_{i=0}^{\infty}\left(r_i/r\right)^\delta=C(\delta)(r/5)^{-\delta}\sum_{i=0}^{\infty}r_i^\delta.\]
Choose $\delta=\beta-q/p>0$, and
combine two previous expressions to obtain the inequality
\begin{equation}\label{e.sums_comp}
\sum_{i=0}^{\infty}r_i^{\beta-q/p}\le  C(c_\mu,\beta,q,p)  r^{\beta-q/p} \sum_{i=0}^{\infty} \frac{r_i^\beta}{\mu(B(z,r_{i}))^{1/p}}\left(\int_{B(z,r_i)} f(y)^p \, d\mu (y)\right)^{1/p}\,,
\end{equation}
which is valid for every $z\in F$.

It follows  from \eqref{e.sums_comp}  
that, for each $z\in F$, there is a ball $B(z, r_{i_z})$ such that $r_{i_z}\le r/5$ and
\[
\mu(B(z,r_{i_z}))r_{i_z}^{-q}\leq C(c_\mu,\beta,q,p)r^{\beta p-q} \int_{B(z, r_{i_z})} f(y)^p \, d \mu (y)\,.
\]
  Using   the $5r$-covering lemma \cite[Lemma~1.7]{MR2867756}, we obtain countably many points $z_k \in F$, $k\in I$,   
such that the  corresponding  balls $B(z_k,r_k)=B(z_k,r_{i_{z_k}})$ are pairwise disjoint and $F \subset \bigcup_{  k\in I} B(z_k,5r_k)$.
Recalling definition \eqref{e.hausdorff_content_density}  of the Hausdorff content of codimension $q$, and using the doubling property of measure $\mu$ and the fact that $\{B(z_k,r_k)\}_{  k\in I}$ is a family of   pairwise   disjoint balls, we obtain
\begin{align*}
\Ha^{\mu,q}_{r}(F)&\leq \sum_{  k\in I} \mu(B(z_k,5r_k))(5r_k)^{-q}\leq c_\mu^3 \sum_{  k\in I} \mu(B(z_k,r_k))r_k^{-q}\\
&\leq C(c_\mu,\beta,q,p) r^{\beta p-q} \sum_{  k\in I} \int_{B(z_k,r_k)} f(y)^p \, d \mu (y) \leq C(c_\mu,\beta,q,p) r^{\beta p-q}\int_{X} f(y)^p \, d \mu (y)\,.
\end{align*}
Applying \eqref{e.hdensity}, we get
\[
 r^{-q}\mu(B(x,r))\le c_1^{-1}\Ha^{\mu, q}_{r}(F)
\le  C(c_1,c_\mu,\beta,q,p)  r^{\beta p-q}\int_{X} f(y)^p \, d \mu (y)\,.
\]
Simplifying the exponents yields
\begin{equation}\label{e.first_case}
r^{-\beta p}\mu(B(x,r))\le  C(c_1,c_\mu,\beta,q,p) \int_{X} f(y)^p \, d \mu (y)\,.
\end{equation}

In the second case, we assume that there exists $z\in F$ such that
\[
\int_{B(z,r/5)} \frac{f(y)d(z,y)^\beta}{\mu(B(z,d(z,y)))}\,d\mu(y)<\frac{1}{2}\,.
\]
By  inequality \eqref{e.X_is_large} and  Lemma \ref{CommonEstimate} we have
\begin{equation}\label{e.second_case2}
r^{-\beta p}\mu(B(x,r))\le 
C(c_\mu,c_\sigma,\beta,\sigma,p) \int_{X} f(y)^p\,d\mu(y)\,. 
\end{equation}

From \eqref{e.first_case} and \eqref{e.second_case2} it follows that
\[
r^{-\beta p}\mu(B(x,r))\le 
C(c_1,c_\mu,c_\sigma,\sigma,\beta,q,p)\int_{X} f(y)^p\,d\mu(y)
\]
for all  nonnegative test functions $f\in L^p(X)$ for the capacity $R_{\beta,p}(F)$. 
By taking infimum over all such functions, we obtain
\[
r^{-\beta p}\mu(B(x,r))\le C(c_1,c_\mu,c_\sigma,\sigma,\beta,q,p)R_{\beta,p}(E\cap\overline{B(x,r)})
\]
for all $x\in E$ and   all   $0<r<\diam(E)/8$, as desired.
\end{proof}

Lemma~\ref{l.hajlasz implies capacity}  leads to  an analogue of Lemma~\ref{l.Hequiv} for the Riesz $(\beta,p)$-capacities. 
 
\begin{lemma}\label{c.balls}
  Assume that $X$ is connected. Let $1<p<\infty$ and $0<\beta\le 1$,
and assume that $\mu$   satisfies the quantitative reverse doubling condition \eqref{e.reverse_doubling}
for some exponent $\sigma>\beta p$. 
Then there is a constant  $C=C(c_\mu,c_\sigma,\sigma,\beta,p)>0$ 
   such that
\[
C^{-1}r^{-\beta p}\mu(B(x,r))\le R_{\beta,p}(\overline{B(x,r)})\le  Cr^{-\beta p}\mu(B(x,r)) 
\]
for all $x\in X$ and all $0<r<(1/8)\diam(X)$. 
\end{lemma}

\begin{proof}
Observe that  $X$ itself  is a closed set
that
satisfies the Hausdorff content density condition \eqref{e.hausdorff_content_density} for every  $q\ge 0$. 
Hence, the first inequality follows from Lemma
\ref{l.hajlasz implies capacity}.
On the other hand, by Theorem \ref{p.HaCap_geq_RieszCap}
and inequality \eqref{e.Hcomp}, we obtain
\[
R_{\beta,p}( \overline{B(x,r)})\le C \cp_{\beta,p} \big(\overline{B(x,r)},B(x,2r) \big)\le C\,r^{-\beta p}\mu(B(x,r))\,.
\]
The second inequality follows.
\end{proof}

 We finish this section with an example showing that the quantitative reverse doubling condition \eqref{e.reverse_doubling} for some exponent $\sigma>\beta p$ cannot be removed from Lemma \ref{c.balls}.  
 
\begin{example}
  Consider   $X=\R^n$ equipped with the Euclidean
distance $d$ and the $n$-di\-men\-sional Lebesgue measure which we denote by $\mu$. Fix $1< p<\infty$ and $0<\beta\le 1$ such that
$\beta p>n$.
 Observe that
the quantitative reverse doubling condition \eqref{e.reverse_doubling}
does not hold in $(\R^n,d,\mu)$ for any $\sigma>\beta p$, since $\beta p>n$.
Define $A_j=B(0,3j)\setminus B(0,2j)$ for all $j=1,2,\ldots$.
Let $z\in \overline{B(0,1)}$. Then
\begin{align*}
I_\beta \left(j^{-\beta}\ch{A_j}\right)(z)
&=j^{-\beta}\int_{A_j}\frac{d(z,y)^\beta}{\mu(B(z,d(z,y)))}\,d\mu(y)
\ge j^{-\beta}j^{\beta}\int_{A_j}\frac{1}{\mu(B(z,d(z,y)))}\,d\mu(y)
\\&\ge \int_{A_j}\frac{1}{\mu(B(0,5j))}\,d\mu(y)= \frac{\mu(A_j)}{\mu(B(0,5j))}= C(n)>0\,.
\end{align*}
It follows that functions $f_j=C(n)^{-1}j^{-\beta}\ch{A_j}$ are admissible
test functions for the Riesz $(\beta,p)$-capacity of $\overline{B(0,1)}$  for every $j=1,2,\ldots$,  and therefore
\[
R_{\beta,p}(\overline{B(0,1)})\le \lVert f_j\rVert_{p}^p
=C(n,p)j^{-\beta p}\mu(A_j) \le C(n,p)j^{n-\beta p}
\]
for all $j=1,2,\ldots$. By taking $j\to \infty$, we see   that  
$R_{\beta,p}(\overline{B(0,1)})=0$. 
\end{example}

\section{Equivalence of density conditions}\label{s.lequiv_density}

  In Corollary~\ref{c.HCapEqvRCap} we showed that the 
Riesz and Haj{\l}asz $(\beta,p)$-capacities
are comparable in a complete and connected space $X$
if $\mu$ satisfies a suitable reverse doubling condition and the kernel estimate~\eqref{e.kernel} holds.
Next, we consider the corresponding capacity density conditions
and show their equivalence in Theorem~\ref{t.Riesz_and_Hajlasz}.
The kernel estimate is no longer needed, but instead we need to assume that
the space $X$ is geodesic since that is required in Theorem~\ref{t.main_characterization}.

The  Haj{\l}asz   capacity density condition is known to be 
self-improving (open-ended) by the results in~\cite{CV2021},
and hence the same also holds for the  Riesz   capacity density condition.
This is stated explicitly in Corollary~\ref{c.riesz-improvement}.

\begin{definition}\label{d.cap_density}
A closed set $E\subset X$ satisfies   the   Haj{\l}asz $(\beta,p)$-capacity density condition,   
for $1\le p<\infty$ and $0<\beta\le 1$, if there
  is   a constant $c_0>0$ such that
\begin{equation}\label{e.Hajlaszc capacity density condition}
\cp_{\beta,p}(E\cap \overline{B(x,r)},B(x,2r))\ge 
 c_0 \: \cp_{\beta,p}(\overline{B(x,r)},B(x,2r)) %
\end{equation}
for all $x\in E$ and all $0<r<(1/8)\diam(E)$.
\end{definition}

It follows from Lemma \ref{l.Hequiv} that a closed set $E$ in a connected metric space $X$ satisfies 
  the   Haj{\l}asz $(\beta,p)$-capacity density condition,   
for $1\le p<\infty$ and $0<\beta\le 1$, if and only if
there is a constant  $c_1>0$  such that
\begin{equation}\label{e.h}
\cp_{\beta,p}(E\cap \overline{B(x,r)},B(x,2r))\ge c_1 \: r^{-\beta p}\mu(B(x,r))
\end{equation}
for all $x\in E$ and all $0<r<(1/8)\diam(E)$. 
 This condition   was taken as the   definition
of the Haj{\l}asz $(\beta,p)$-capacity density condition
in \cite{CV2021}. 

 The following   result   from \cite[  Theorem~9.5]{CV2021} shows that in  complete geodesic metric spaces 
  the   Haj\l asz capacity density condition is equivalent to  a  Hausdorff content density condition. 

\begin{theorem}\label{t.main_characterization}
  Assume that $X$ is   a complete geodesic space.  Let
$1<p<\infty$ and $0<\beta\le 1$, and let   $E\subset X$ be a closed 
set. Then the following conditions are equivalent
\begin{enumerate}
\item[\textup{(i)}] $E$ satisfies the Haj{\l}asz $(\beta,p)$-capacity density condition \eqref{e.Hajlaszc capacity density condition}.
\item[\textup{(ii)}] $E$ satisfies the Hausdorff content density condition \eqref{e.hausdorff_content_density} for some $0 < q < \beta p$.
\end{enumerate}
\end{theorem}

  Theorem~\ref{t.main_characterization} and Lemma~\ref{l.hajlasz implies capacity} are the main
ingredients in the proof of Theorem~\ref{t.Riesz_and_Hajlasz}, which 
is one of our main results and adds the following Riesz capacity density condition
to the list of conditions in Theorem~\ref{t.main_characterization}. 
See also~\cite[Theorem~9.5]{CV2021} for many more equivalent conditions,
and Theorem~\ref{t.main_wide_beta=1} for the case $\beta=1$.  

\begin{definition}\label{d.riesz_cap_density}
A closed set $E\subset X$ satisfies   the   Riesz $(\beta,p)$-capacity density condition,   
for $1\le p<\infty$ and $\beta >0$, if there
  is   a constant $c_0>0$ such that
\begin{equation}\label{e.Riesz capacity density condition}
R_{\beta,p}(E\cap \overline{B(x,r)})\ge c_0 \: R_{\beta,p}(\overline{B(x,r)}) 
\end{equation}
 for all $x\in E$ and all $0<r<(1/8)\diam(E)$.
\end{definition}

\begin{theorem}\label{t.Riesz_and_Hajlasz}
  Assume that $X$ is   a complete geodesic space.   Let
$1<p<\infty$ and $0<\beta\le 1$, and
assume that $\mu$   satisfies the quantitative reverse doubling condition \eqref{e.reverse_doubling}
for some exponent $\sigma>\beta p$. Then the following
conditions are equivalent   for a closed set $E\subset X$:  
\begin{enumerate}
\item[\textup{(i)}] $E$ satisfies the Riesz $(\beta,p)$-capacity density condition \eqref{e.Riesz capacity density condition}.
\item[\textup{(ii)}] $E$ satisfies the Haj{\l}asz $(\beta,p)$-capacity density condition \eqref{e.Hajlaszc capacity density condition}.
\item[\textup{(iii)}] $E$ satisfies the Hausdorff content density condition \eqref{e.hausdorff_content_density} for some $0 < q < \beta p$.
\end{enumerate}
\end{theorem}

\begin{proof}
  Since $X$ is geodesic, it is in particular connected.  
To show the implication from (i) to (ii) we assume that $E$ satisfies the Riesz $(\beta,p)$-capacity density condition \eqref{e.Riesz capacity density condition}. By  using this condition, Theorem~\ref{p.HaCap_geq_RieszCap}
and Lemma~\ref{c.balls}, we find that inequality \eqref{e.h} holds
  for all $x\in E$ and all $0<r<(1/8)\diam(E)$.   
  By Lemma~\ref{l.Hequiv}, this   is equivalent to the Haj\l asz $(\beta,p)$-capacity density condition.

The implication from (ii) to (iii) is a part of Theorem~\ref{t.main_characterization}, and   
the implication from (iii) to (i) follows by Lemma~\ref{l.hajlasz implies capacity} and Lemma~\ref{c.balls}.
\end{proof}

From Theorem~\ref{t.Riesz_and_Hajlasz} 
we obtain as a corollary
 that the Riesz capacity density condition is self-improving. 
  This   can be regarded as an extension of Lewis' result \cite{MR946438}
to the setting of   complete geodesic spaces.   
For the proof of the corollary, it suffices
to notice that condition (iii) in Theorem~\ref{t.Riesz_and_Hajlasz}
remains valid even if we modify both $\beta$ and $p$ slightly.
  In the Euclidean case, this approach to the self-improvement 
of Riesz capacity density condition, which is largely
based on the results in~\cite{CV2021}, is
completely different  from   the original argument of Lewis.

\begin{corollary}\label{c.riesz-improvement}
  Assume that $X$ is   a complete geodesic space.    Let
$1<p<\infty$ and $0<\beta\le 1$, and
assume that $\mu$   satisfies the quantitative reverse doubling condition \eqref{e.reverse_doubling}
for some exponent $\sigma>\beta p$.
  If   %
a closed set $E\subset X$ satisfies the Riesz $(\beta,p)$-capacity density condition,
  then   there exists $0<\delta<\min\{\beta,p-1\}$ such that 
$E$ satisfies the Riesz $(\gamma,s)$-capacity density condition
  whenever  
$\beta-\delta<\gamma\le 1$,   $p-\delta<s<\infty$   and   $\sigma > \gamma s$.
\end{corollary}

\section{Comparability of Haj{\l}asz \texorpdfstring{$(1,p)$}{(1,p)}-capacity and \texorpdfstring{$p$}{p}-capacity}\label{s.equivbeta1}

  In this last section of the paper, we add into the considerations 
one more notion of capacity, the variational $p$-capacity.
The definition of this capacity is given in terms of 
$p$-weak upper gradients, and we begin by recalling relevant preliminaries.
The main result of this section is Theorem~\ref{t.comparison of capacities}, 
which shows the comparability of the variational $p$-capacity and the $(1,p)$-Haj{\l}asz capacity 
under suitable assumptions. The comparability of the capacities implies
also the equivalence of the corresponding density conditions, stated 
at the end of the section in Corollary~\ref{c.pcap_and_hajlasz},
which in turn can be combined with the results in Section~\ref{s.lequiv_density}
to give further lists of equivalent conditions in the case $\beta=1$;
see Theorems~\ref{t.main_wide_beta=1} and~\ref{t.omnibus}.

  Let $u$ be    a real valued function on $X$. 
  A Borel function $g\ge 0$ on $X$ is  
an {\em upper gradient} of $u$   if   for all curves $\gamma$  
  (see Section \ref{s.metric})   joining 
any two points $x,y \in X$ we have
\begin{equation}\label{e.modulus}
\lvert u(x)-u(y)\rvert \le \int_\gamma g\,ds\,,
\end{equation}
whenever both $u(x)$ and $u(y)$ are finite, and $\int_\gamma g\,ds =\infty$
otherwise. 
In addition, a measurable function $g\ge 0$ on $X$ is a {\em $p$-weak upper gradient}
of   $u$, for $1\le p<\infty$,    
if inequality~\eqref{e.modulus} holds for $p$-almost every curve $\gamma$ joining arbitrary points $x$ and $y$ in $X$. That is, there exists a nonnegative Borel function $\rho\in L^p_{\textup{loc}}(X)$ such that 
$\int_\gamma \rho\,ds=\infty$ whenever~\eqref{e.modulus} does not hold for the curve $\gamma$.
 Here  $\rho\in L^p_{\textup{loc}}(X)$ means  that
for each $x\in X$ there exists $r_x>0$ such that $\rho\in L^p(B(x,r_x))$.
We refer to~\cite{MR2867756} for more information on $p$-weak upper gradients.

We say that the space $X$ supports a {\em $p$-Poincar\'e  inequality}, for $1\le p<\infty$,  
if there exist constants $C_P>0$ and $\lambda\ge 1$
such that for all balls $B\subset X$, all measurable
functions $u$ on $X$, and all $p$-weak upper gradients $g$ of $u$   we have  
\begin{equation}\label{e.poincare}
 \vint_B\lvert u-u_B\rvert\,d\mu 
\le C_P\, \diam(B) \left(\vint_{\lambda B} g^p\,d\mu\right)^{1/p}\,.  
\end{equation} 
 Here $u_B$ is the integral average of $u$ over $B$ as in \eqref{e.intaverage},  and the left-hand side of~\eqref{e.poincare} is interpreted as $\infty$ whenever $u_B$ is not defined. 
We refer to \cite[Chapter~4]{MR2867756} for further details. 
  The doubling condition~\eqref{e.doubling} and 
the Poincar\'e inequality are the standard assumptions on analysis on 
metric spaces based on (weak) upper gradients;
recall that we assume throughout that $\mu$
is doubling with a constant $c_\mu$.

\begin{definition}
  Let $1\le p<\infty$ and let $\Omega\subset X$  be a bounded open set.   
The variational $p$-capacity of a  closed   subset $F\subset\Omega$ is 
\begin{equation}\label{e.capacity}
\cp_p(F,\Omega) = \inf_u\inf_g \int_\Omega g(x)^p\,d\mu(x)\,,
\end{equation}
where the infimums are taken over all $u\in \Lip(X)$, with $u\ge 1$ in $F$, $u=0$ in $X\setminus \Omega$, 
and all $p$-weak upper gradients $g$ of $u$. If there are no such functions $u$, we set $\cp_p(F,\Omega)=\infty$.  
\end{definition}

\begin{remark}\label{e.equiv_cap}
If $\cp_p(F,\Omega)<\infty$, 
then the infimum in~\eqref{e.capacity} can be restricted to
functions $u\in \Lip(X)$ satisfying $\ch{F} \le u\le \ch{\Omega}$
and to $p$-weak upper gradients $g$ of $u$ such that $g=g\ch{\Omega}\in L^p(X)$.
For a proof we refer to \cite[Remark 2.1]{MR3673660}. 
\end{remark}

\begin{definition}\label{d.capacity density}
  A   closed set $E\subset X$ satisfies   the   $p$-capacity density condition,
for $1\le p<\infty$, if there
  is   a constant $c_0>0$ such that
\begin{equation}\label{e.capacity density}
\cp_p(E\cap \overline{B(x,r)},B(x,2r))\ge c_0 \cp_p(\overline{B(x,r)}, B(x,2r))
\end{equation}
for all $x\in E$ and all $0<r<(1/8)\diam(E)$.
\end{definition}

\begin{remark}\label{r.Capacity and measure} 
If $X$ supports a $p$-Poincar\'e inequality for some $1\le p<\infty$, 
then there is   a close   connection between   the   $p$-capacity and the measure
  of balls, similar to Lemma~\ref{l.Hequiv}.   
Namely, there is a constant $C>0$ such that for all
balls $B=B(x,r)$, with $0<r<(1/8)\diam(X)$, and for 
  all closed sets   $F\subset \overline{B}$   we have  
\begin{equation}\label{e.comparison}
\frac{\mu(F)}{C\, r^p}\le \cp_p(F,2B)\le \frac{c_\mu \,\mu(B)}{r^p}\,.
\end{equation} 
For a proof of this fact see, for instance, \cite[Proposition~6.16]{MR2867756}. 
In particular,   it holds 
for all balls $B=B(x,r)$ with $0<r<(1/8)\diam(X)$ that  
\begin{equation}\label{e.capacity of a ball}
\frac{\mu(B)}{C \: r^p }\leq  \cp_p\big( \overline B, 2B \big) \leq   \frac{c_\mu\:\mu(B)}{r^p}\,.
\end{equation}
As a consequence, if $E\subset X$ is a closed set and $X$ supports
a $p$-Poincar\'e inequality, then $E$ satisfies   the   $p$-capacity density condition~\eqref{e.capacity density} if and only if 
there is a constant $c_1>0$ such that
 
\[
\cp_p(E\cap \overline{B(x,r)},B(x,2r))\ge c_1 r^{-p}\mu(B(x,r))
\]
for all $x\in E$ and all $0<r<(1/8)\diam(E)$.
\end{remark}

  The following Theorem~\ref{t.comparison of capacities} says that for $1<p<\infty$ the variational   
$p$-capacity and $(1,p)$-Haj{\l}asz capacity  are comparable in the appropriate geometrical setting.
In the proof we use the noncentered maximal function $Mf$, which is defined by \eqref{e.max_funct_def}.

\begin{theorem}\label{t.comparison of capacities}
  Let   %
$1<p<\infty$. Assume that $\Omega \subset X$ is a bounded open set and 
$F\subset \Omega$ is a closed set. Then  
\begin{equation}
\label{e.inequality one}
 \cp_{p} (F,\Omega) \leq  4^p   \cp_{1,p}(F,\Omega)\,.
\end{equation}
Moreover, if $X$ supports  a $q$-Poincar\'e inequality  for some $1\leq q <p$,   with constants $C_P$ and $\lambda\ge 1$,  then there   is a constant  $C=C(C_P,c_\mu,p,q)$  such that 
\begin{equation}
\label{e.inequality two}
 \cp_{1,p} (F,\Omega) \leq C  \cp_p(F,\Omega)\,.
\end{equation}
\end{theorem}

\begin{proof}%
  There are test functions for $\cp_{p} (F,\Omega)$ if and only if there are test functions for
$\cp_{1,p} (F,\Omega)$. Namely, the existence
of test functions   is in both cases   characterized by the inequality $\dist(F,X\setminus \Omega)>0$. 
Without loss of generality, we may assume that this inequality holds   since otherwise
both   capacities are equal to $\infty$.

We begin with the proof of inequality \eqref{e.inequality one}. Let $u$ be a test function for $\cp_{1,p}(F,\Omega)$, that is,
$u\in \Lip(X)$ with $u \ge1 $ in $F$ and $u=0$ in $X\setminus \Omega$.
  Let $\kappa\ge 0$ be a Lipschitz constant of $u$ and let   
$g$ be a Haj{\l}asz $  1  $-gradient of $u$.  
By redefining $g=\kappa$ in the exceptional set $N=N(g)$ of measure zero, we may assume that
inequality~\eqref{e.hajlasz} holds for all $x,y\in X$, with $\beta=1$. Then arguing as in 
\cite[Lemma~4.7]{MR1809341} with the aid of \cite[p.~20]{MR2867756}, we see that $4g$ is actually a $p$-weak upper gradient of $u$.   Therefore,
\begin{align*}
\cp_p(F,\Omega) \leq \int_\Omega (4 g)^p\, d\mu 
& \leq 4^p \int_X g^p\, d\mu\,.
\end{align*}
Since this holds for all $u$ and all Haj{\l}asz $  1  $-gradients $g$ of $u$, we conclude that \eqref{e.inequality one} holds.

Next we prove   estimate \eqref{e.inequality two}
under the assumption that   $X$ supports a $q$-Poincar\'e inequality for 
some $1\le q<p$, with constants $C_P$ and $\lambda\ge 1$. 
 Let $u\in \Lip(X)$ be a test function for $\cp_p(F,\Omega)$ and let $g$ be a $p$-weak upper gradient of $u$. 
Observe that $g$ is also a $q$-weak upper gradient of $u$, since $q<p$.  
By Remark \ref{e.equiv_cap}, we may assume that $g=g \ch{\Omega}$.

We claim that a   constant   multiple of $(Mg^q)^{  1/q}$ is a $1$-Haj{\l}asz gradient of $u$, that is, 
\begin{equation}\label{e.Mh is Hajlasz gradient}
|u(x)-u(y)| \leq  C(C_P,c_\mu)   \:d(x,y)  \Bigl( \big(Mg^q(x)\big)^{  1/q} + \big(Mg^q(y)\big)^{  1/q} \Bigr)
\end{equation}
for all $x,y\in X$.    
To prove \eqref{e.Mh is Hajlasz gradient},  
we follow a  chaining argument from \cite[p.~13--14]{MR1683160}.  
For this purpose, we let $x,y\in X$, $x\not=y$, and   define   %
$B_0^x = B_0^y=B\big(x,2d(x,y)\big)$ and 
\[
B_j^x = B\big(x,2^{-j} d(x,y)\big), \qquad B_j^y = B\big(y,2^{-j} d(x,y)\big), \qquad j\geq 1\,. 
\]
Notice that the definition   of   $B_0^x=B_0^y$ is different 
 from the other ones.  
Nevertheless, the balls  are nested in the following way: $B_0^x  \supset  B_1^x \supset B_2^x \supset \dotsb$ and $B_0^y \supset B_1^y \supset B_2^y \supset \dotsb$. Moreover,   we have $\mu(B_j^x) \leq c_\mu^2 \mu(B_{j+1}^x)$ and $\mu(B_j^y) \leq c_\mu^3 \mu(B_{j+1}^y)$ for all $j=0,1,\ldots$ by the doubling property \eqref{e.doubling}.    Using  continuity of $u$  and the above properties, we have  
\begin{align*}
|u(x)-u(y)| & \leq |u(x)-u_{B_0^x}| + |u(y)-u_{B_0^y}|\\
& \leq \sum_{j=0}^\infty |u_{B_{j+1}^x}- u_{B_j^x}| + \sum_{j=0}^\infty |u_{B_{j+1}^y}- u_{B_j^y}| \\
& \leq \sum_{j=0}^\infty \vint_{B_{j+1}^x} |u-u_{B_j^x}|\, d\mu \:+ \: \sum_{j=0}^\infty \vint_{B_{j+1}^y} |u-u_{B_j^y}|\, d\mu \\
& \leq \sum_{j=0}^\infty \frac{\mu(B_j^x)}{\mu(B_{j+1}^x)}\vint_{B_{j}^x} |u-u_{B_j^x}|\, d\mu \:+ \: \sum_{j=0}^\infty \frac{\mu(B_j^y)}{\mu(B_{j+1}^y)} \vint_{B_{j}^y} |u-u_{B_j^y}|\, d\mu \\
& \leq c_\mu^2 \sum_{j=0}^\infty \vint_{B_{j}^x} |u-u_{B_j^x}|\, d\mu \:+ \:   c_\mu^3   \sum_{j=0}^\infty  \vint_{B_{j}^y} |u-u_{B_j^y}|\, d\mu\,.
\end{align*}
   
By applying the assumed $q$-Poincar\'e inequality to the pair $u$ and $g$, we obtain
\begin{align*}
&|u(x)  - u(y)| \\
&\quad \leq c_\mu^2   C_P \sum_{j=0}^\infty   \diam(B_j^x)  \biggl(  \vint_{  \lambda B_{j}^x } g^q \,d\mu\biggr)^{  1/q} \:+ \:  c_\mu^3C_P   \sum_{j=0}^\infty   \diam(B_j^y)   \biggl(  \vint_{  \lambda B_{j}^y } g^q \,d\mu\biggr)^{  1/q} \\
&\quad \leq c_\mu^2   C_P \: \bigl(Mg^q(x)\bigr)^{  1/q} \sum_{j=0}^\infty   \diam(B_j^x)   \: + \:   c_\mu^3 C_P   \: \bigl(Mg^q(y)\bigr)^{  1/q} \sum_{j=0}^\infty   \diam(B_j^y)  \\
&\quad \leq   C(C_P,c_\mu)   \: d(x,y) \: \Bigl( \bigl(Mg^q(x)\bigr)^{  1/q} + \bigl( Mg^q(y)\bigr)^{  1/q} \Bigr)\,.
\end{align*}
Here we use   also   the facts that $x\in \lambda B_j^x$ and $y\in \lambda B_j^y$ for all $j=0,1,\ldots$. 
This works even for the ball $\lambda B_0^y$ that is not centered at $y$, but contains $y$. 
This proves the claim \eqref{e.Mh is Hajlasz gradient}.  

 Since $C(C_P,c_\mu)(Mg^q)^{  1/q}\in\mathcal{D}_H^1(u)$
  and the maximal operator is bounded on $L^{  p/q}(X)$,   
we   obtain for the Haj{\l}asz $(1,p)$-capacity the estimate  
\begin{equation}\label{e.end_cap}
\begin{split}
 \cp_{1,p}(F,\Omega)
& \leq   C(C_P,c_\mu,p,q)  \int_X \big(M g^q \big)^{  p/q} \, d\mu \\
& \leq    C(C_P, c_\mu, p,q)   \int_X g ^p \, d\mu 
 =    C(C_P, c_\mu, p,q)     \int_{\Omega} g^p \, d\mu\,,
 \end{split}
 \end{equation}
where we   also use the fact that   $g=g\ch{\Omega}$.   
Since estimate  \eqref{e.end_cap} holds for all test functions $u$ 
for $\cp_p(F,\Omega)$ 
and all their $p$-weak upper gradients $g$   satisfying  $g=g\ch{\Omega}$, 
we conclude that \eqref{e.inequality two} holds. 
  This completes the proof.  
 \end{proof}
 
  Next we give an example which shows   that
the $q$-Poincar\'e inequality assumption cannot be omitted
  from the second part of Theorem~\ref{t.comparison of capacities}, i.e.\  
for inequality \eqref{e.inequality two} to hold.
This example also shows that the $p$-Poincar\'e inequality assumption cannot be omitted
for the first inequality in \eqref{e.capacity of a ball} to hold.

 \begin{example}\label{shows}
 Fix $1<p<q<\infty$ and let
$w(x)=\dist(x,\{1,-1\})^{q-1}$ for all $x\in\R$. 
We consider the metric measure
space $X=\R$ equipped with the Euclidean distance $d$ and
the weighted measure $\mu$ such that \[\mu(A)=\int_A w(x)\,dx\]
for all Borel sets $A\subset \R$.
By
\cite[Theorem 10.26]{KLV2021} we see that $w$ belongs to 
the Muckenhoupt class $A_{q+\varepsilon}$ for all $\varepsilon>0$. In particular, the
measure $\mu$ is doubling.
Fix $0<\rho<1$
and consider the Lipschitz test
function \[u(x)=\max\left\{0,1-\frac{\dist(x,[-1,1])}{\rho}\right\}\,,\qquad x\in\R\,.\]
  Then   %
$\lvert u'\rvert$ is a $p$-weak upper gradient of $u$,
  by \cite[Proposition~1.14]{MR2867756},
and so  
\begin{align*}
\cp_p(\overline{B(0,1)},B(0,2))
&\le \int_{B(0,2)} \lvert u'(x)\rvert^p w(x)\,dx
\le 2\rho^{-p}\int_0^\rho t^{q-1}\,dt =\frac{2}{q}\rho^{q-p}\,.
\end{align*}
By taking $\rho\to 0$, we find that
\[
\cp_p(\overline{B(0,1)},B(0,2))=0\,.
\]
Hence, the first inequality in \eqref{e.capacity of a ball}  cannot hold.
  Nevertheless, Remark~\ref{l.Hequiv} implies   that
\[\cp_{1,p}(\overline{B(0,1)},B(0,2))\ge C\,\mu(B(0,1))>0\] and therefore 
inequality \eqref{e.inequality two} cannot hold for any $C_2>0$.
From \cite[Theorem~2]{BjornBuckleyKeith2006} and \cite[Theorem~10.26]{KLV2021} it follows 
that $(X,d,\mu)$ 
supports a $(q+\varepsilon)$-Poincar\'e inequality for all $\varepsilon>0$
and that $(X,d,\mu)$ does not support a $q$-Poincar\'e inequality.
In particular, $(X,d,\mu)$ does not support a $(p-\varepsilon)$-Poincar\'e inequality for any $\varepsilon>0$. 
\end{example}

  The equivalence of the variational 
$p$-capacity and the $(1,p)$-Haj{\l}asz capacity
in Theorem~\ref{t.comparison of capacities}
gives immediately also the equivalence of the corresponding
density conditions, provided the  
space supports a suitable Poincar\'e inequality.

\begin{corollary}\label{c.pcap_and_hajlasz}
Let $1\le q < p<\infty$ and assume that $X$ supports a $q$-Poincar\'e inequality. 
Then   a closed set $E\subset X$   satisfies
  the   Haj{\l}asz $(1,p)$-capacity density condition if and only if $E$ satisfies   the   $p$-capacity density condition.
\end{corollary}

  By combining Corollary~\ref{c.pcap_and_hajlasz}
and Theorem~\ref{t.Riesz_and_Hajlasz}, we obtain the
following characterization of the $p$-capacity density condition in terms of the  
Riesz capacity, Haj{\l}asz capacity and Hausdorff content density conditions.

\begin{theorem}\label{t.main_wide_beta=1}
  Let $1<p<\infty$ and
assume that $X$ is a complete geodesic space 
supporting a $p$-Poincar\'e inequality. In addition, assume  
that $\mu$   satisfies the quantitative reverse doubling condition \eqref{e.reverse_doubling}
for some exponent $\sigma>p$.
Then the following conditions are equivalent
  for a closed set $E\subset X$:  
\begin{itemize}
\item[\textup{(i)}] $E$ satisfies the $p$-capacity density condition.
\item[\textup{(ii)}] $E$ satisfies the Riesz $(1,p)$-capacity density condition.
\item[\textup{(iii)}] $E$ satisfies the Haj{\l}asz $(1,p)$-capacity density condition.
\item[\textup{(iv)}] $E$ satisfies the Hausdorff content density condition \eqref{e.hausdorff_content_density} for some $0 < q < p$.
\end{itemize}
\end{theorem}

\begin{proof}
Since $X$ is complete and  $\mu$ is doubling, the  Keith--Zhong theorem 
\cite{MR2415381} implies
that $X$ supports an $s$-Poincar\'e inequality for
some $1<s<p$.   
The conditions (i) and (iii) are then  equivalent by   Corollary~\ref{c.pcap_and_hajlasz}.  
  On the other hand, the equivalence of the conditions (ii), (iii) and (iv) follows from Theorem~\ref{t.Riesz_and_Hajlasz}.  
\end{proof}

A similar argument using Theorem~\ref{t.main_characterization}
instead of Theorem~\ref{t.Riesz_and_Hajlasz} gives the following result, which excludes the Riesz capacity density condition
but does not assume the quantitative reverse doubling condition.

\begin{theorem}\label{t.omnibus}
  Let $1<p<\infty$ and
assume that $X$ is a complete geodesic space 
supporting a $p$-Poincar\'e inequality.  
Then the following conditions are equivalent
  for a closed set $E\subset X$:  
\begin{itemize}
\item[\textup{(i)}] $E$ satisfies the $p$-capacity density condition.
\item[\textup{(ii)}] $E$ satisfies the Haj{\l}asz $(1,p)$-capacity density condition.
\item[\textup{(iii)}] $E$ satisfies the Hausdorff content density condition \eqref{e.hausdorff_content_density} for some $0 < q < p$.
\end{itemize}
\end{theorem}

  Since the Haj{\l}asz $(1,p)$-capacity density condition is self-improving with respect to $p$
in geodesic spaces, the above results imply the self-improvement also for the
$p$-capacity density condition.  
The following corollary is known, see \cite{MR1869615}, 
but our   method of   proof ,
based on the results from~\cite{CV2021},   is new.  

\begin{corollary}
  Let $1<p<\infty$ and
assume that $X$ is a complete geodesic space 
supporting a  $p$-Poincar\'e inequality. 
Let $E\subset X$ be a closed set satisfying   the   $p$-capacity density condition. 
Then $E$ satisfies   the   $q$-capacity density condition
for some $1<q<p$. 
\end{corollary}

\begin{proof}
By the Keith--Zhong theorem \cite{MR2415381}, we see that $X$ supports
an $s$-Poincar\'e inequality for some $1<s<p$. Hence,
the conclusion follows from Theorem \ref{t.omnibus}.
\end{proof}

\bibliographystyle{abbrv}
\def\cprime{$'$}

\setlength{\parindent}{0pt}

\end{document}